\documentclass[preprint,12pt]{elsarticle}

\usepackage{graphics}

\usepackage{amssymb}
\usepackage{amsthm}
\usepackage{amsmath}
\usepackage{mathabx}
\usepackage{color}
\usepackage{amsopn}
\usepackage{amsmath}
\usepackage{mathabx}
\usepackage{pstricks}
\usepackage{color}
\usepackage{pstcol}
\usepackage{pst-plot}
\usepackage{pst-tree}
\usepackage{pst-eps}
\usepackage{multido}
\usepackage{pst-node}
\usepackage{pst-3d}
\usepackage{colordvi}
\usepackage{multirow}
\usepackage{pst-grad}
\usepackage{colortbl}

\DeclareMathOperator{\interior}{Int}

\newtheorem{theorem}{Theorem}[section]

\newtheorem{definition}[theorem]{Definition}

\newtheorem{lemma}[theorem]{Lemma}
\newtheorem{proposition}[theorem]{Proposition}
\newtheorem{remark}[theorem]{Remark}

\journal{Journal of Multivariate Analysis}

\begin{document}

\renewcommand{\listtablename}{\'indice de tablas}

\nocite{*}

\begin{frontmatter}



\title{A fully data-driven method for estimating the shape of a point cloud}

\author[rvt]{A. Rodr\'iguez-Casal}
\author[rvt]{P. Saavedra-Nieves \corref{cor1}}
\address[rvt]{Department of Statistics and Operations Research, University of Santiago de Compostela, Spain}
\cortext[cor1]{Corresponding author: paula.saavedra@usc.es (P. Saavedra-Nieves)}

\begin{abstract}
Given a random sample of points from some unknown distribution, we propose a new data-driven method for estimating its probability support $S$.
Under the mild assumption that $S$ is $r-$convex, the smallest
$r-$convex set which contains the sample points is the natural estimator. The main problem for using this estimator in practice
is that $r$ is an unknown geometric characteristic of the set $S$. A stochastic algorithm is proposed
for selecting its optimal value from the data under the hypothesis that the sample is uniformly generated. The new data-driven reconstruction of $S$ is able to
achieve the same convergence rates as the convex hull for estimating convex sets, but under a much more flexible
smoothness shape condition. 
\end{abstract}

\begin{keyword}Support estimation \sep $r-$convexity \sep uniformity \sep maximal spacing
\end{keyword}

\end{frontmatter}



\section{Introduction}
\label{intro}

Support estimation deals with the problem of reconstructing the
compact and nonempty support $S\subset \mathbb{R}^d$ of an absolutely continuous random vector $X$ assuming that a
random sample $\mathcal{X}_{n}=\{X_{1},...,X_{n}\}$ from $X$ is given. In practical terms, the question is
how to reconstruct the contour of Aral Sea in Figure \ref{kkkk090} from the uniform sample $\mathcal{X}_{1500}$ drawn inside it.

\begin{figure}
\begin{picture}(-30,80)
\put(-7,-56){\includegraphics[height=4.35cm, width=.223\textwidth]{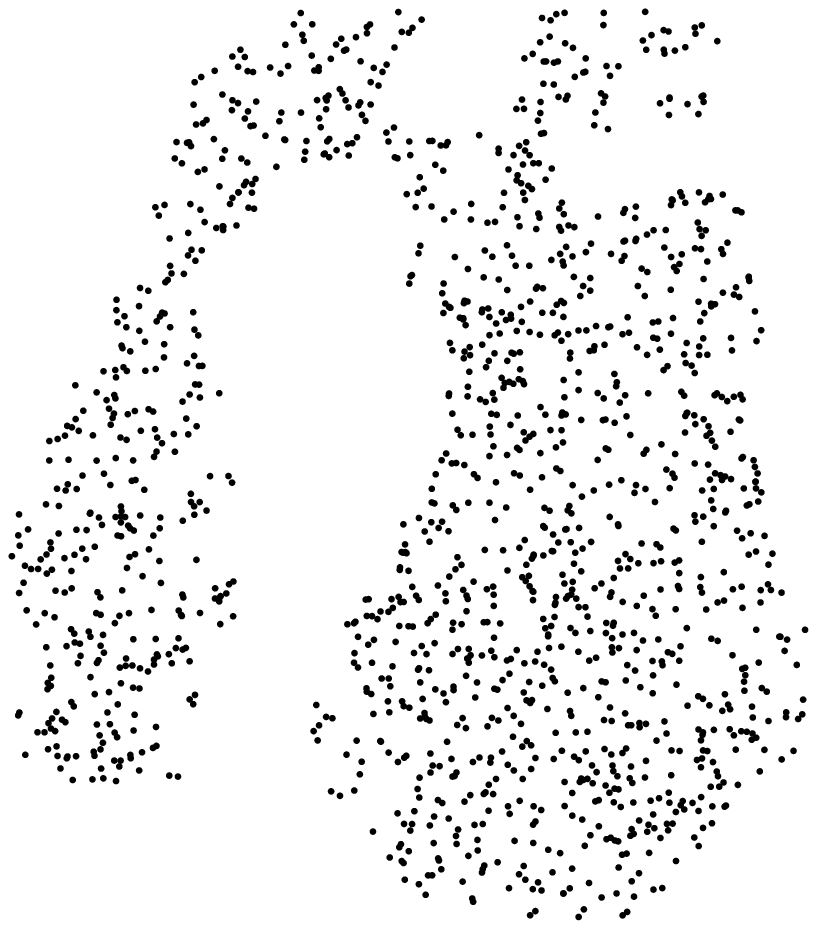}}
\put(143,-40){\includegraphics[height=3.02cm, width=.223\textwidth]{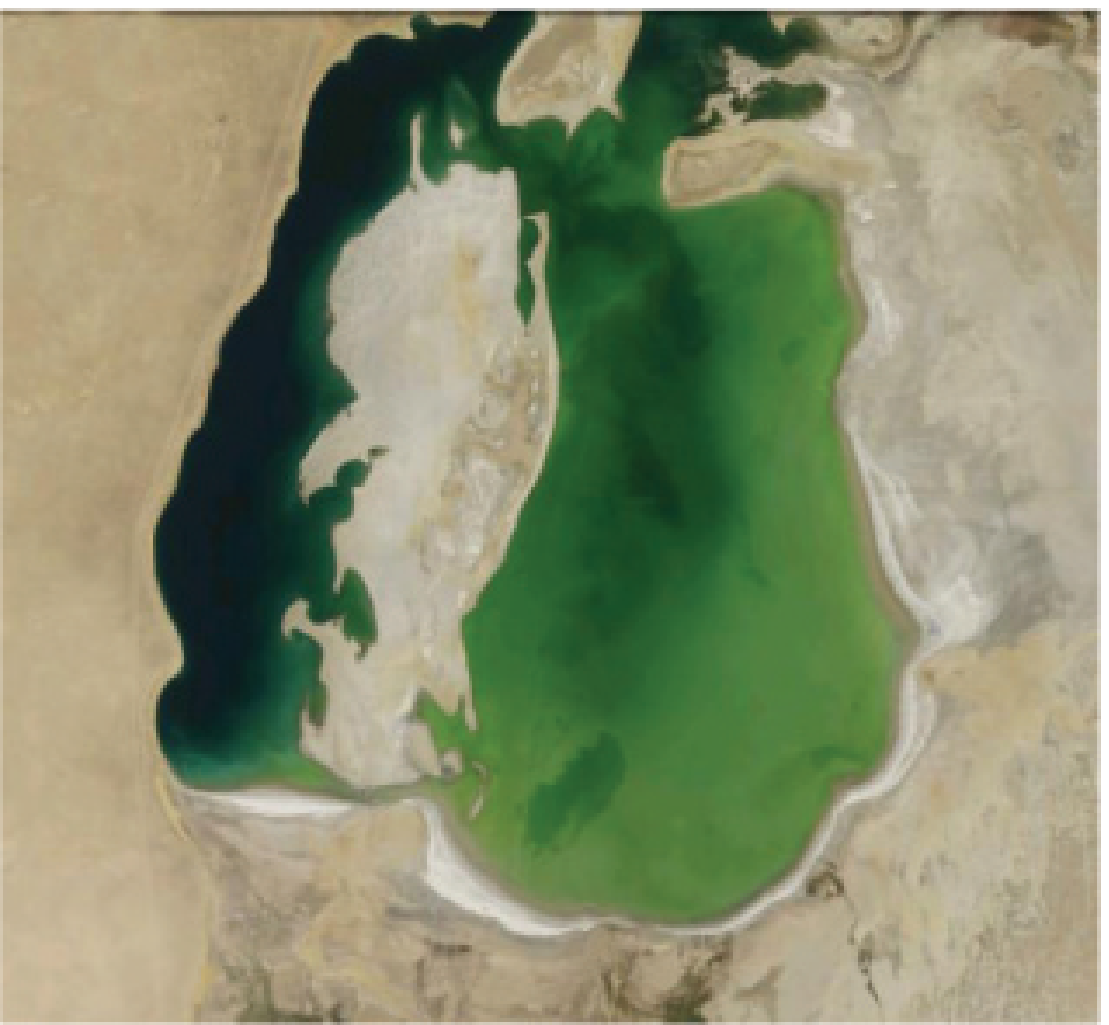}}
\put(305,-68){\includegraphics[height=5.1cm, width=.26\textwidth]{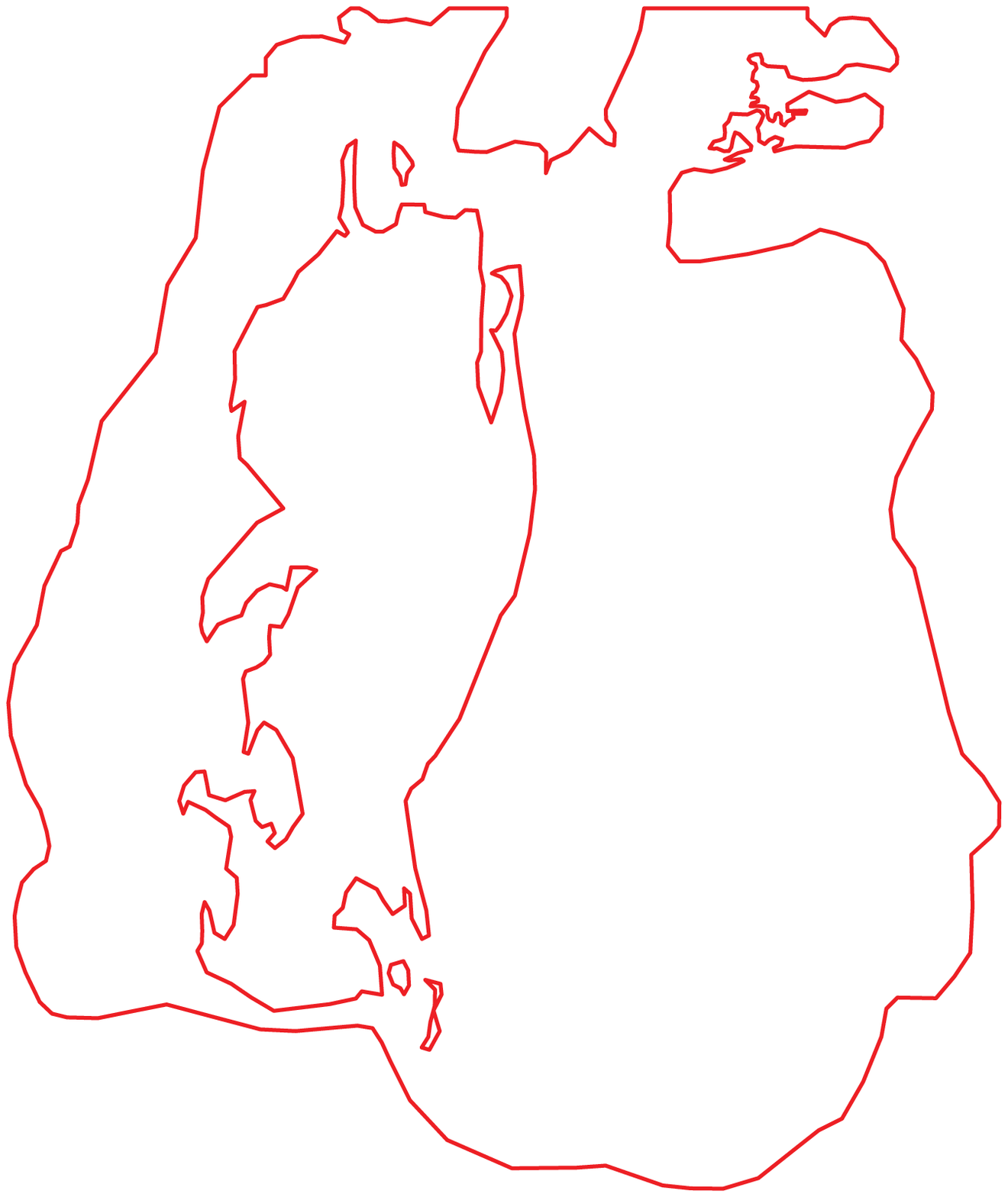}}
\put(40,66){(a)}
\put(182,66){(b)}
\put(352,66){(c)}
\end{picture}
\vspace{1.7cm}\caption{(a) $\mathcal{X}_{1500}$ on the Aral Sea. (b) Aral Sea's image from the Moderate Resolution Imaging Spectroradiometer on NASA's Terra satellite in 2000. (c) Aral Sea's boundary.}\label{kkkk090}
\end{figure}

The previous question has different but quite natural responses depending on the available information on $S$. For example, if no assumptions
are made a priori on the shape of the support $S$, Chevalier (1976) and Devroye and Wise (1980) proposed a
general purpose estimator which is just a sort of \emph{dilated} version of $\mathcal{X}_{n}$. Specifically,
$$
S_n=\bigcup_{i=1}^nB_{\epsilon_n}[X_i],
$$where $B_{\epsilon_n}[X_i]$ denotes the closed ball centered at $X_i$ with radius $\epsilon_n$, a sequence of
smoothing parameters which must tend to zero but not too quickly in order to achieve consistency. See also Grenander (1981), Cuevas (1990), Korostel\"ev and Tsybakov (1993) or Cuevas
and Rodr\'iguez-Casal (2004). The main disadvantage of this estimator is its dependence on the unknown and influential radius of the balls $\epsilon_n$. Small values of $\epsilon_n$ provide split estimators whereas for large values of $\epsilon_n$ the estimator could considerably overestimate $S$. Ba\'illo et al. (2000) and Ba\'illo and Cuevas (2001) suggested two general methods
for selecting the parameter $\epsilon_n$ assuming that $S$ is connected and star-shaped, respectively.

However, more sophisticated alternatives, that can achieve better error rates, could be used if some a priori
information about the shape of $S$ is available. For instance, if the support is assumed to be convex then the convex hull of the
sample points, $H(\mathcal{X}_n)$, provides a natural support estimator. This
is just the intersection of all convex sets containing $\mathcal{X}_{n}$. For analyzing in depth this estimator, see
Schneider (1988, 1993), D\"{u}mbgen and Walther (1996) or Reitzner (2003).

In practise, the convexity assumption may be too restrictive, see the Aral Sea example in Figure \ref{kkkk090}. So, it can be useful to introduce the notion of $r-$convexity, a more
flexible shape condition. A closed set $A\subset\mathbb{R}^d$ is said to be $r-$convex, for some $r>0$, if $A=C_{r}(A)$, where
$$C_{r}(A)=\bigcap_{\{B_r(x):B_r(x)\cap
A=\emptyset\}}\left(B_r(x)\right)^c$$
denotes the $r-$convex hull of $A$ and $B_r(x)$, the open ball with
center $x$ and radius $r$. The $r-$convex hull is closely related to the closing of $A$ by $B_r(0)$ from the mathematical morphology, see Serra (1982). It can be shown that
$$
C_{r}(A)=(A\oplus r B)\ominus r B,
$$
where $B=B_1(0)$, $\lambda C=\{\lambda c: c\in C\}$, $C\oplus D=\{c+d:\ c\in C, d\in D\}$ and $C\ominus D=\{x\in\mathbb{R}^d:\ \{x\}\oplus D\subset C\}$, for $\lambda \in \mathbb{R}$ and
sets $C$ and $D$.

 If it is assumed that $S$ is $r-$convex, $C_r(\mathcal{X}_n)$ is the natural estimator for the support. This estimator is well known in the computational geometry literature for producing
good global reconstructions if the sample points are (approximately) uniformly distributed on the
set $S$. See Edelsbrunner (2014) for a survey on the subject. Although the
 $r-$convexity is a more general restriction than the convexity, $C_r(\mathcal{X}_n)$ can achieve the same convergence rates than $H(\mathcal{X}_n)$, see Rodr\'{\i}guez-Casal (2007). However, this
 estimator depends on the unknown parameter $r$. Figure \ref{oooooolgiop} shows its influence by
 using the random sample on the Aral Sea presented in Figure \ref{kkkk090}. Small values of $r$ provide estimators almost equal to $\mathcal{X}_n$. However,
 if large values of $r$ are considered then $C_r(\mathcal{X}_n)$ practically coincides with $H(\mathcal{X}_n)$, see Figure \ref{oooooolgiop} (d).

\begin{figure}
\hspace{-.2cm}\begin{picture}(-30,33)
\put(-10,-79){\includegraphics[height=3.1cm, width=.23\textwidth]{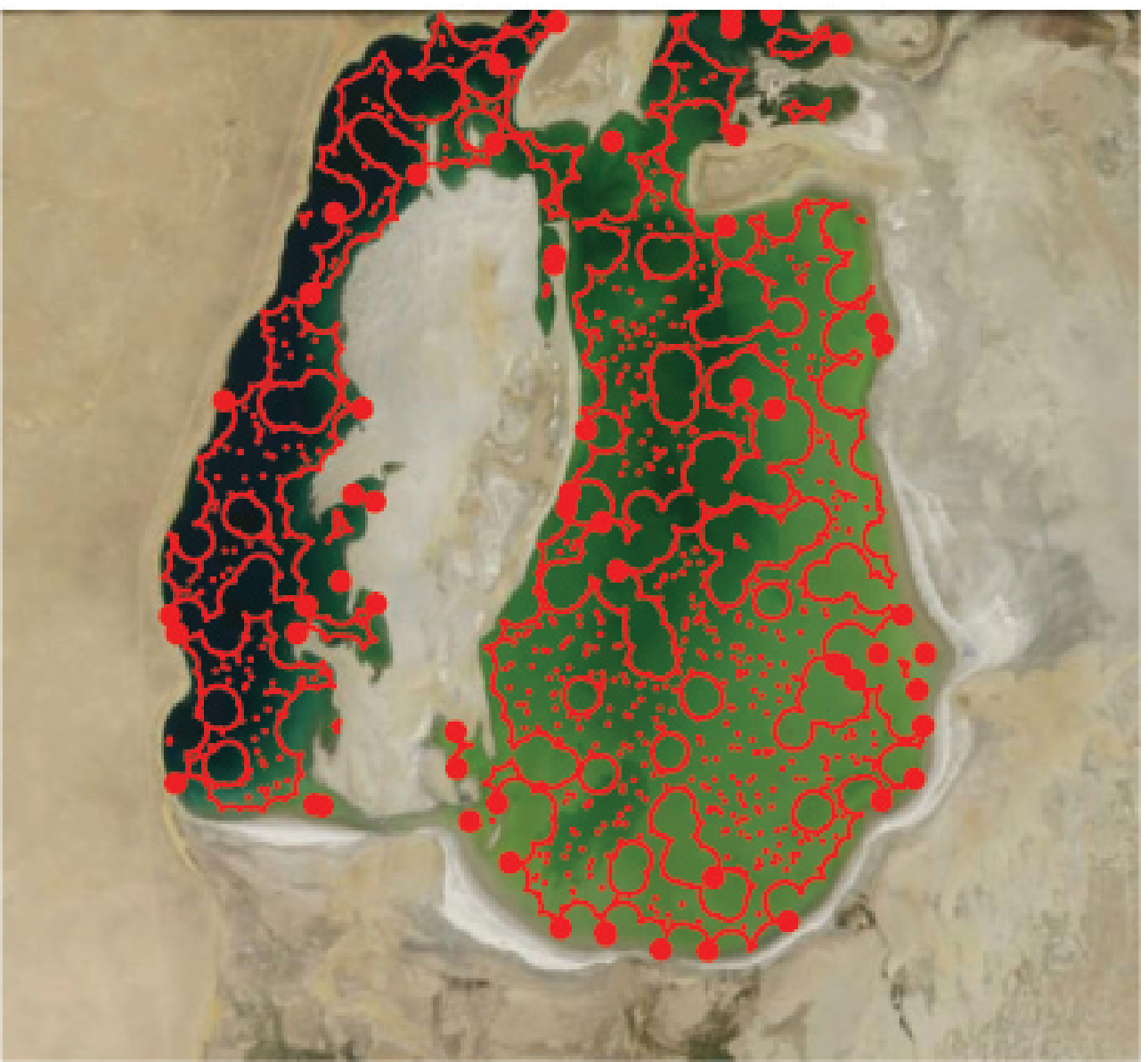}}
\put(100,-79){\includegraphics[height=3.1cm, width=.23\textwidth]{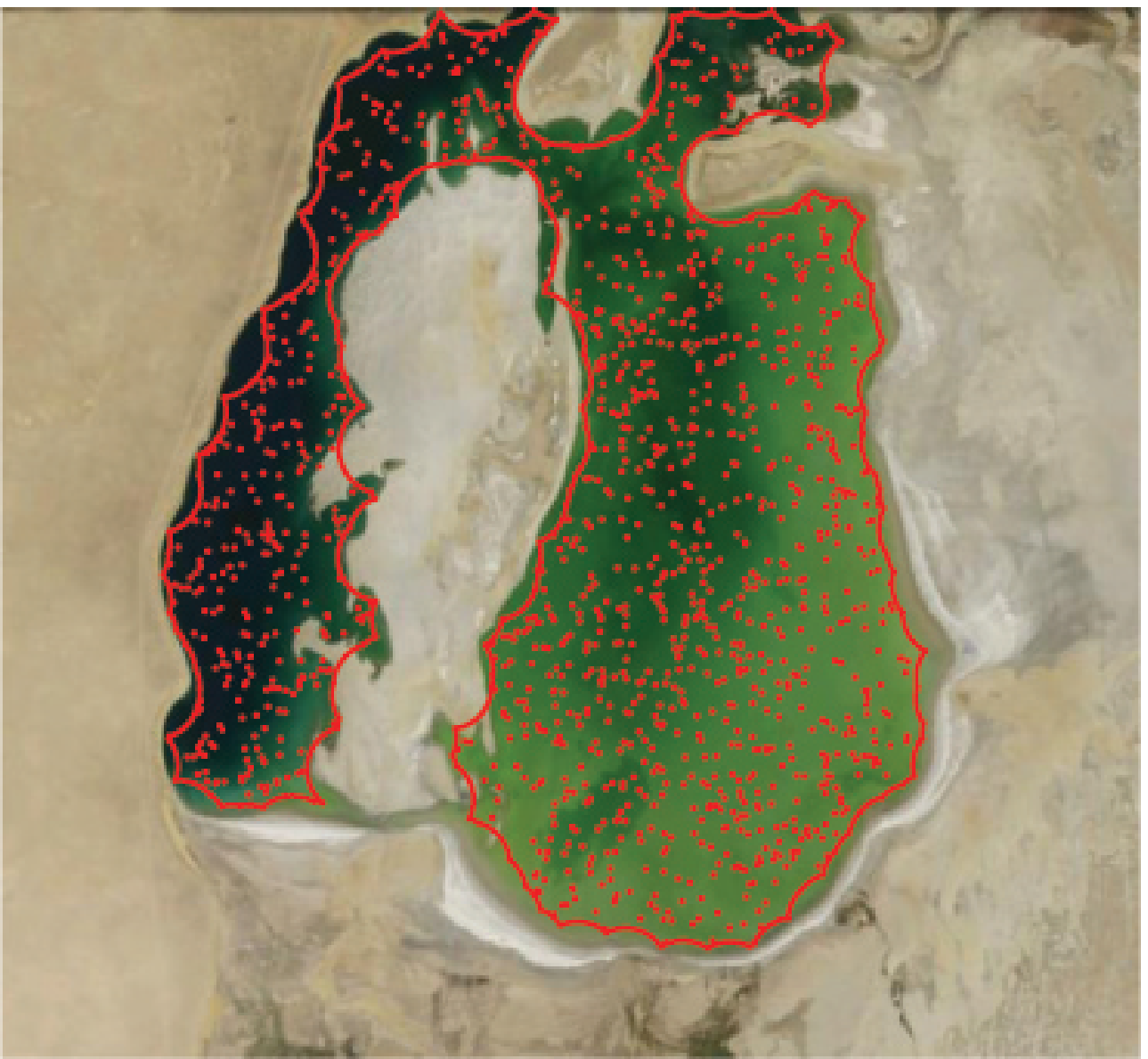}}
\put(210,-79){\includegraphics[height=3.1cm, width=.23\textwidth]{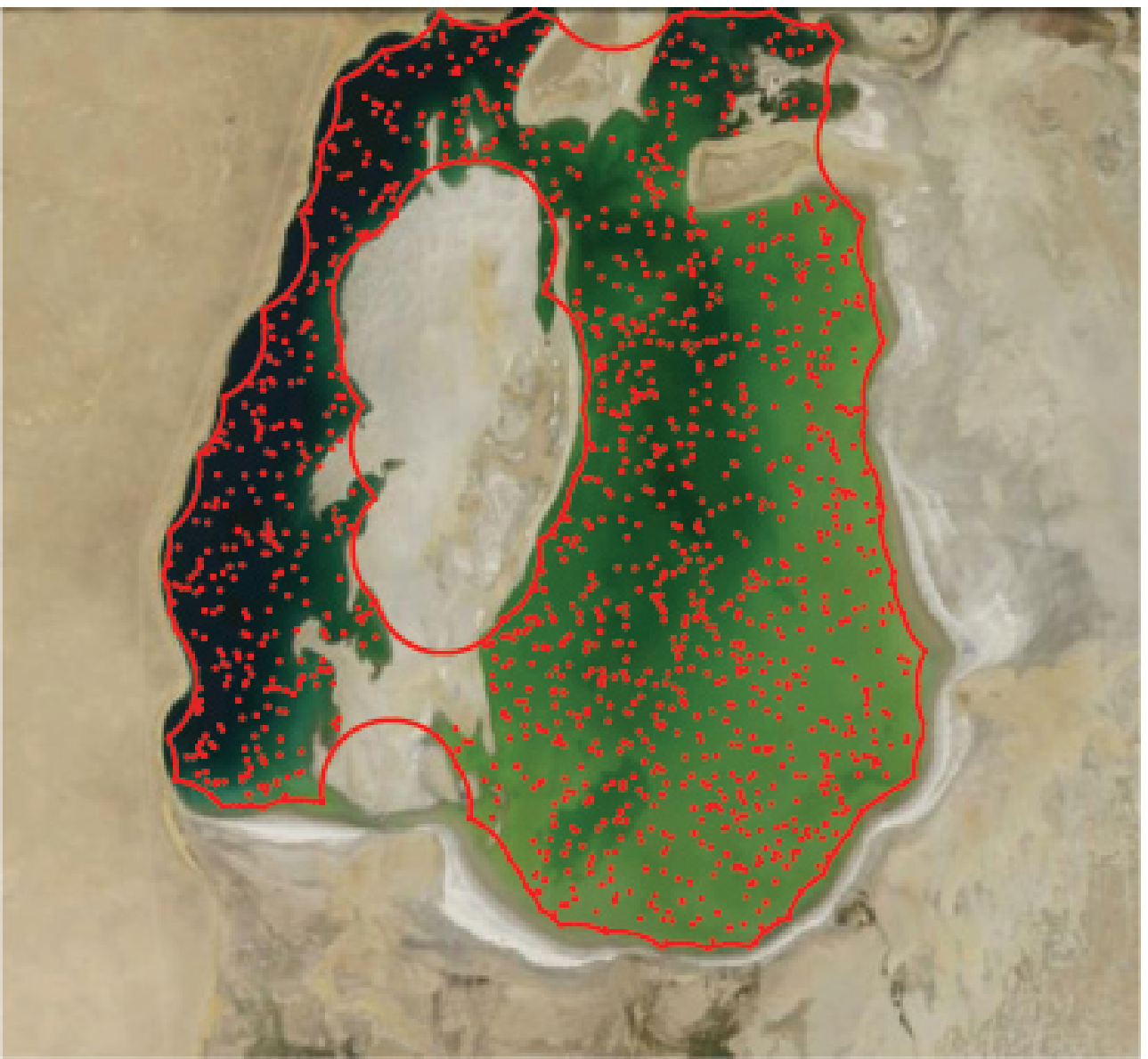}}
\put(320,-79){\includegraphics[height=3.1cm, width=.23\textwidth]{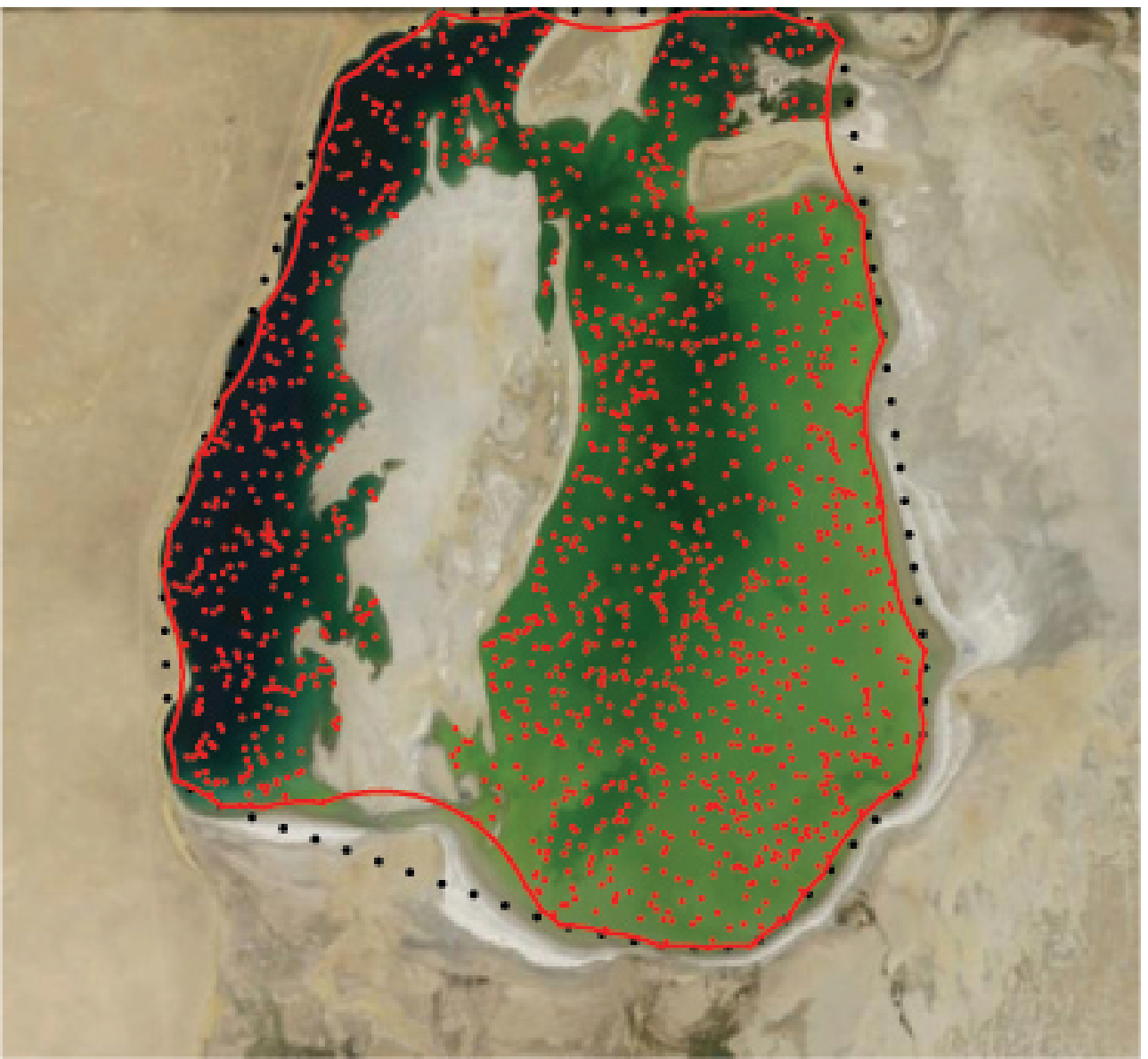}}
\put(30,25){(a)}
\put(140,25){(b)}
\put(248,25){(c)}
\put(359,25){(d)}
\end{picture}
\vspace{3cm}\caption{The boundary of $C_r(\mathcal{X}_{1500})$ is shown in red for (a) $r=10$, (b) $r= 25$, (c) $r=40$ and (d) $r=90$. The boundary of $H(\mathcal{X}_{1500})$ is shown in dotted line in (d).}\label{oooooolgiop}
\end{figure}

According to the previous comments, support estimation can be considered as a geometric counterpart of the classical theory
of nonparametric functional estimation, see Simonoff (1996). The estimators typically depend on a sequence of
smoothing parameters in both theories. Theoretical results make special emphasis on asymptotic properties, especially
consistency and convergence rates but they do not give any criterion for selecting the unknown parameters.
The aim of this paper is to overcome this drawback and present a method for selecting the parameter $r$ from the available data. This problem, for the bidimensional case,
has already been studied in literature by Mandal and Murthy (1997). They proposed a selector for $r$ based on
the concept of minimum spanning tree but only consistency of the method was provided.

The automatic selection criterion which will be proposed
in this work is based on a very intuitive idea. As it can be seen in Figure \ref{oooooolgiop} (c) or (d), land areas are contained
in $C_r(\mathcal{X}_n)$ if $r$ is too large. So, the estimator contains a big ball (or spacing) empty of sample points. Janson (1987) calibrated the size of the maximal spacing
when the sample distribution is uniform on $S$. Recently, Berrendero et al. (2012) used this result to test uniformity when the support is unknown. Here, we will
follow the somewhat opposite approach. We will assume that $\mathcal{X}_n$ follows a uniform distribution on $S$ and if a big
enough spacing is found in $C_{r}(\mathcal{X}_n)$ then $r$ is too large. We select the largest value of $r$ compatible with the uniformity assumption on $C_{r}(\mathcal{X}_n)$.

Once the parameter $r$ is estimated, it is natural to go back to the support estimation problem. An automatic estimator for $S$, based on the estimator of $r$, is proposed in this paper. Two metrics between sets are usually considered in order to assess the performance of a set estimator. Let $A$
 and $C$ be two closed, bounded, nonempty subsets of $\mathbb{R}^{d}$. The
Hausdorff distance between $A$ and $C$ is defined by\vspace{-0.13cm}
$$
d_{H}(A,C)=\max\left\{\sup_{a\in A}d(a,C),\sup_{c\in C}d(c,A)\right\},\vspace{-0.13cm}
$$where $d(a,C)=\inf\{\|a-c\|:c\in C\}$ and $\|\mbox{ }\|$ denotes the Euclidean norm. On the other hand, if $A$ and $C$
are two bounded and Borel sets then the distance in measure between $A$ and $C$ is defined by $d_{\mu}(A,C)=\mu(A\triangle C)$, where $\mu$ denotes the Lebesgue measure and $\triangle$, the symmetric difference, that is, $A\triangle C=(A \setminus C)\cup(C \setminus A). $
Hausdorff distance quantifies the physical proximity between two sets whereas the distance in measure
is useful to quantify their similarity in content. However, neither of these distances are completely useful for measuring
the similarity between the shape of two sets. The Hausdorff distance between boundaries, $d_H(\partial A,\partial C)$, can be also used to evaluate the
performance of the estimators, see Ba\'illo and Cuevas (2001), Cuevas and Rodr\'iguez-Casal (2004) or Rodr\'iguez-Casal (2007).

This paper is organized as follows. In Section \ref{nosometedo}, the optimal smoothing parameter of $C_r(\mathcal{X}_n)$ to be estimated is established. The new data-driven
algorithm for selecting it is presented in Section \ref{resultadosmaximo2}. Consistency of this
estimator is established in Section \ref{soooooo}. In addition, a
new estimator for the support $S$ is proposed. It is showed that it is able to achieve the same convergence rates as the convex hull for estimating convex sets. The numerical questions involving the practical application of the algorithm are analyzed in
Section \ref{numerical}. In Section \ref{simulation}, the performances of the new selector and Mandal and Murthy (1997)'s method will be
analyzed through a simulation study. Finally, proofs are deferred to Section \ref{prprpr}.

\section{Optimal smoothing parameter for the $r-$convex hull}\label{nosometedo}


The problem of reconstructing a $r-$convex support $S$ using a data-driven procedure can be solved if the parameter $r$ is estimated from a
random sample of points $\mathcal{X}_n$ taken in $S$. Next, it will be presented an algorithm to do this. The first step is to determine precisely
the value of $r$ to be estimated. It is established in Definition \ref{sup}. We propose to estimate the highest value of $r$ which verifies that $S$ is $r-$convex.

\begin{definition}\label{sup}Let $S\subset \mathbb{R}^{d}$ a compact, nonconvex and  $r-$convex set for some $r>0$. It is defined
\begin{equation}\label{maximo2}
r_0=\sup\{\gamma>0:C_\gamma(S)=S\}.\end{equation}
\end{definition}

For simplicity in the exposition, it is assumed that $S$ is not convex. Of course, if $S$ is convex $r_0$ would be infinity. In Proposition \ref{maximo}, it is proved that the supreme established in (\ref{maximo2}) is a maximum of the set $\{\gamma>0:C_\gamma(S)=S\}$. Therefore, it is possible to guarantee that $S$ is $r_0-$convex too. Then, the optimality of the smoothing parameter defined in (\ref{maximo2}) can be justified. For $r<r_0$, $C_{r}(\mathcal{X}_n)$ is a non admisible estimator since it is always outperformed by $C_{r_0}(\mathcal{X}_n)$. This is because, with probability one,
$C_{r}(\mathcal{X}_n)\subset C_{r_0}(\mathcal{X}_n)\subset S$ and hence, $d_{\mu}(C_{r_0}(\mathcal{X}_n),S)\leq d_{\mu}(C_{r}(\mathcal{X}_n),S)$ (the same holds for the Hausdorff distance). On the other hand, if
$r>r_0$ then $C_{r}(\mathcal{X}_n)$ would considerably overestimate $S$ specially if $S$ has a big hole inside, see Figure \ref{oooofffffdfffg} (a) below. However, it is not enough to assume that $S$ is $r-$convex for obtaining the proof of Proposition \ref{maximo}. It was necessary to suppose that $S$ satisfies a new geometric property slightly stronger than $r-$convexity: \vspace{4mm}\\
($R_{\lambda}^r\label{new}$) $S$ fulfills the $r-$rolling property and $S^c$ fulfills the $\lambda-$rolling condition.\vspace{3mm}\\
Following Cuevas et al. (2012), it is said $A$ satisfies the (outside) $r-$rolling condition if each boundary point $a\in\partial A$ is contained in a closed ball with radius $r$ whose interior does not meet $A$. The intuitive concept of rolling freely can be seen as a sort of geometric smoothness statement that is preserved if the limit is considered, see Proposition \ref{nonece}. There exist interesting relationships between this property and $r-$convexity. In particular, Cuevas et al. (2012) proved that if $A$ is compact and $r-$convex then $A$ fulfills the $r-$rolling condition. According to Figure \ref{Figura1}, the reciprocal is not true. For a in depth analysis of these two shape restrictions see Walther (1997, 1999).

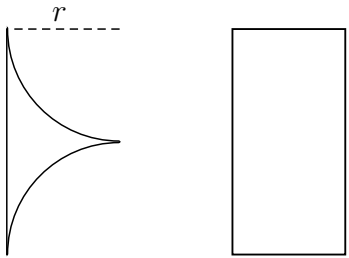
\begin{figure}[h!]
\begin{pspicture}(-3.2,-.2)(5,4.5)
\pspolygon[fillstyle=solid,fillcolor=white,linecolor=black,linewidth=0.2mm] (1.5,0)(1.5,3)(3,3)(3,0)
\pscircle[fillstyle=solid,fillcolor=white,linecolor=black,linewidth=0.2mm,linearc=3](3,3){1.5}
\pscircle[fillstyle=solid,fillcolor=white,linecolor=black,linewidth=0.2mm,linearc=3](3,0){1.5}
\pswedge[fillstyle=solid,linewidth=0.2mm,linecolor=white](3.031,1.5){3.3}{225.09}{134.6}
\pspolygon[fillstyle=solid,fillcolor=white,linecolor=black,linewidth=0.25mm] (4.5,0)(4.5,3)(6,3)(6,0)
\psline[linearc=0.25,linestyle=dashed,dash=3pt 2pt,linecolor=black,linewidth=0.2mm](1.6,3)(3,3)
\rput(2.2,3.2){$r$}
\end{pspicture}\caption{$A$ fulfills the $r-$rolling condition $\nRightarrow$ $A$ is $r-$convex.}\label{Figura1}
\end{figure}

 \begin{proposition}\label{nonece}Let $A\subset\mathbb{R}^{d}$ be a closed set. Let $\{r_n\}$ be a sequence of positive terms converging to $\overline{r}$. If
 $A$ fulfills the $r_n-$rolling condition, for all $n$, then $A$ fulfills the $\overline{r}-$rolling condition.\end{proposition}

Sets satisfying condition ($R_{\lambda}^r$) have a number of desirable properties which make them easier to handle and more general than the class of sets considered in Walther (1997, 1999) where only the case $r=\lambda$ is taken into account. In this work, the radius $\lambda$ can be different from $r$, see Figure \ref{Figura22}. Walther (1997, 1999) proved that, under ($R_{r}^r$),  $S$ is $r-$convex. In Proposition \ref{rconvexo}, it will be proved that, under ($R_{\lambda}^r$) for any value $\lambda>0$, $S$ is $r-$convex too. Therefore, ($R_{\lambda}^r$) is a sufficient condition for guaranteeing $r-$convexity of the support $S$; however, ($R_{\lambda}^r$) is not a necessary condition. Figure \ref{Figura1111} shows three $r-$convex sets which do not satisfy ($R_{\lambda}^r$) for any $\lambda>0$. As conclusion and according to the previous comments, under ($R_\lambda^r$), the equivalence between $r-$convexity and rolling property for radius $r$ can be obtained taking into account Proposition \ref{rconvexo}.

\begin{figure}[h!]\centering
\scalebox{.75}[.75]{
\begin{pspicture}(0,-1)(10,6)
\psarc[showpoints=false,linearc=0.25,linecolor=black,linewidth=0.2mm](0,0){.75}{180}{270}
\psarc[showpoints=false,linearc=0.25,linecolor=black,linewidth=0.2mm](3,0){.75}{270}{90}
\psarc[showpoints=false,linearc=0.25,linecolor=black,linewidth=0.2mm](3,3.7){.75}{270}{90}
\psarc[showpoints=false,linearc=0.25,linecolor=black,linewidth=0.2mm](0,3.7){.75}{90}{180}
\psline[linearc=0.25,linecolor=black,linewidth=0.2mm](-0.75,0)(-0.75,3.7)
\psline[linearc=0.25,linecolor=black,linewidth=0.2mm](3.,4.45)(0,4.45)
\psline[linearc=0.25,linecolor=black,linewidth=0.2mm](3.,-.75)(-0.,-.75)
\pscircle[linearc=0.25,linecolor=black,linewidth=0.2mm](0,0){.75}
\psdots*[dotsize=2.5pt](0,0)
\psline[linearc=0.25,linecolor=black,linewidth=0.2mm](0.,0)(0.75,0)
\rput(.35,.23){$\lambda$}
\psarc[showpoints=false,linearc=0.25,linecolor=black,linewidth=0.2mm](3.1,1.85){1.1}{94}{266}
\pscircle[linearc=0.25,linecolor=black,linewidth=0.2mm](3.1,1.85){1.1}
\psdots*[dotsize=2.5pt](3.1,1.85)
\psline[linearc=0.25,linecolor=black,linewidth=0.2mm](3.1,1.85)(4.2,1.85)
\rput(3.6,2.05){$r$}
\psarc[showpoints=false,linearc=0.25,linecolor=black,linewidth=0.2mm](7,0){.75}{180}{270}
\psarc[showpoints=false,linearc=0.25,linecolor=black,linewidth=0.2mm](10,0){.75}{270}{90}
\psarc[showpoints=false,linearc=0.25,linecolor=black,linewidth=0.2mm](10,3.7){.75}{270}{90}
\psarc[showpoints=false,linearc=0.25,linecolor=black,linewidth=0.2mm](7,3.7){.75}{90}{180}
\psline[linearc=0.25,linecolor=black,linewidth=0.2mm](6.25,0)(6.25,3.7)
\psline[linearc=0.25,linecolor=black,linewidth=0.2mm](10.,4.45)(7,4.45)
\psline[linearc=0.25,linecolor=black,linewidth=0.2mm](10.,-.75)(7,-.75)
\pscircle[linearc=0.25,linecolor=black,linewidth=0.2mm](7,0){.75}
\psdots*[dotsize=2.5pt](7,0)
\psline[linearc=0.25,linecolor=black,linewidth=0.2mm](7.,0)(7.75,0)
\rput(7.35,.23){$\lambda$}
\psarc[showpoints=false,linearc=0.25,linecolor=black,linewidth=0.2mm](10.1,1.85){1.1}{94}{266}
\pscircle[linearc=0.25,linecolor=black,linewidth=0.2mm](9.75,1.85){.75}
\psdots*[dotsize=2.5pt](9.75,1.85)
\psline[linearc=0.25,linecolor=black,linewidth=0.2mm](9.75,1.85)(10.5,1.85)
\rput(10.1,2.1){$\lambda$}
 \end{pspicture}}
\caption{($R_{\lambda}^r$) is a more general condition.}\label{Figura22}
\end{figure}
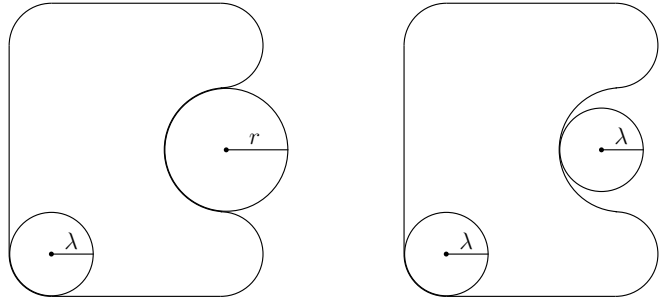$$\vspace{-18mm} $$

\begin{figure}[h!]\centering
\begin{pspicture}(-1,-.2)(6,5)
\pspolygon[fillstyle=solid,fillcolor=white,linecolor=black,linewidth=0.2mm] (-.5,0)(-.5,3)(-2.5,1.5)
\pspolygon[fillstyle=solid,fillcolor=white,linecolor=black,linewidth=0.2mm] (1.5,0)(1.5,3)(3,3)(3,0)
\pscircle[fillstyle=solid,fillcolor=white,linecolor=black,linewidth=0.2mm,linearc=3](3,3){1.5}
\pscircle[fillstyle=solid,fillcolor=white,linecolor=black,linewidth=0.2mm,linearc=3](3,0){1.5}
\pswedge[fillstyle=solid,linewidth=0.2mm,linecolor=white](3.031,1.5){3.3}{225.09}{134.6}
\psline[linearc=0.25,linestyle=dashed,dash=3pt 2pt,linecolor=black,linewidth=0.2mm](1.6,3)(3,3)
\psline[fillstyle=solid,fillcolor=white,linecolor=black,linewidth=0.2mm] (5,3)(5,0)(7,0)(7,3)
  \psarc[showpoints=false,linearc=0.2,linecolor=black,linewidth=0.2mm,fillcolor=white](6,3){1}{180}{360}
  \psline[linearc=0.25,linestyle=dashed,dash=3pt 2pt,linecolor=black,linewidth=0.2mm](5,3)(6,3)
\rput(2.2,3.2){$r$}
\rput(5.5,3.2){$r$}
\rput(-1.67,3.99){$A_1$}
\rput(2.3,3.99){$A_2$}
\rput(6.,3.99){$A_3$}
\end{pspicture}\caption{$A_1$ is convex and, so, $r-$convex for all $r>0$. $A_2$ and $A_3$ are $r-$convex.}\label{Figura1111}
\end{figure}
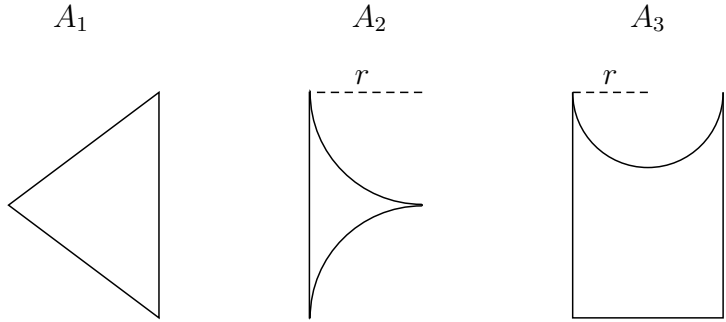

\begin{proposition}\label{rconvexo}Let $S\subset \mathbb{R}^{d}$ be a nonempty, compact support verifying ($R_{\lambda}^r$). Then, $S$ is $r-$convex.\end{proposition}

Having presented the relationships between the different geometric conditions, we are now ready to prove that the supreme defined in (\ref{maximo2}) is a maximum.

\begin{proposition}\label{maximo}Let $S\subset \mathbb{R}^{d}$ be a nonempty, compact and nonconvex set verifying ($R_{\lambda}^r$) and let $r_0$ be the parameter defined in (\ref{maximo2}). Then, $C_{r_0}(S)=S$ and, as consequence, $S$ fulfills the $r_0-$rolling condition.
\end{proposition}

But, why is not this property true if it is assumed that $S$ is only $r-$ convex? Under ($R_\lambda^r$), we have proved that if $\{r_n\}$ converges to $r_0$ and $C_{r_n}(S)=S$ then it is verified that $C_{r_0}(S)=S$, see Proposition \ref{maximo}. Obviously, if $C_{r_n}(S)=S$ then $S$ would be $r_n-$convex, for all $ r_n\in\{r_n\}$. Then, $S$ satisfies the $r_n-$rolling condition, for all $ r_n\in\{r_n\}$. So and according to Proposition \ref{nonece}, $S$ satisfies the $r_0-$rolling condition too. However, taking into account only the rolling property with radius $r_0$ is not enough to guarantee that $S$ is $r_0-$convex.

\section{Selection of the optimal smoothing parameter}\label{resultadosmaximo2}

The uniformity test proposed in Berrendero et al. (2012) has been considered in order to estimate $r_0$ defined
in (\ref{maximo2}) from $\mathcal{X}_n$. This test is based on the multivariate spacings, see Janson (1987). In the univariate case, spacings are defined as
the length of gaps between sample points, $\mathcal{X}_n$. For general dimension $d$, the maximal spacing of $S$ is defined as
$$\Delta_n(S)=\sup\{\gamma:\exists x\mbox{ with }B_\gamma[x]\subset S\setminus \mathcal{X}_n\}.$$
The value of the maximal spacing depends only on $S$ and on the sample points $\mathcal{X}_n$. The Lebesgue measure (volume) of the balls
with radius $\Delta_n(S)$ is denoted by $V_n(S)$. Berrendero et al. (2012) used the Janson (1987)'s Theorem to introduce a uniformity test on $S$. They consider the problem of testing
$$H_0:\mbox{ }X\mbox{ is uniform with support }S.$$With significance level $\alpha$, $H_0$ will be rejected if
 \begin{equation}\label{55}
    V_n(S)>\frac{a(u_\alpha+\log{n}+(d-1)\log{\log{n}}+\log{\beta})}{n},
\end{equation}
where $a=\mu(S)$, $u_\alpha$ denotes the $1-\alpha$ quantile of a random variable $U$ with distribution
\begin{equation}
\label{U}
\mathbb{P}(U\leq u) =\exp(-\exp(-u)) \mbox{ for } u\in \mathbb{R}
\end{equation}
and the value of the constant $\beta$ is explicitly given in Janson (1987). For instance, $\beta=1$ for the bidimensional case.
If $S$ is unknown this test can not be directly applicable. Under the ($R_{\lambda}^r$) condition with $\lambda=r$, Berrendero et al. (2012) considered $S_n=C_r(\mathcal{X}_n)$ as the estimator of $S$, but no data-driven method was provided for selecting $r$. The maximal spacing of $S$ is estimated by
\begin{equation}\label{11}\hat{\Delta}_n=\sup\{\gamma:\exists x\mbox{ with }B_\gamma[x]\subset S_n\setminus \mathcal{X}_n\},\end{equation}and the critical region (\ref{55}) can be replaced
by
$$
    \hat{V}_{n,r}>\hat{c}_{n,\alpha,r}=\frac{a_n(u_\alpha+\log{n}+(d-1)\log{\log{n}}+\log{\beta})}{n},
$$where $a_n=\mu(C_{r}(\mathcal{X}_n))$ and $\hat{V}_{n,r}$ denotes the volume of the ball of radius $\hat{\Delta}_n$, see (\ref{11}).


\begin{figure}
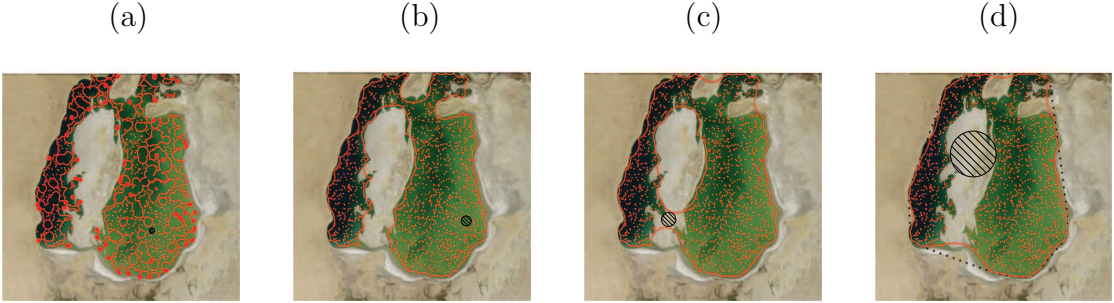

\hspace{-.2cm}\begin{picture}(-30,33)
\put(-10,-79){\includegraphics[height=3.1cm, width=.23\textwidth]{Envolturarconvexa_10bcn1500_222.eps}}
\put(100,-79){\includegraphics[height=3.1cm, width=.23\textwidth]{Envolturarconvexa_25bcn1500_222.eps}}
\put(210,-79){\includegraphics[height=3.1cm, width=.23\textwidth]{Envolturarconvexa_40bcn1500_222.eps}}
\put(320,-79){\includegraphics[height=3.1cm, width=.23\textwidth]{Envolturarconvexa_90bcn1500_222punteada.eps}}
\put(30,25){(a)}
\put(140,25){(b)}
\put(248,25){(c)}
\put(359,25){(d)}
\pscircle[linewidth=.2pt,fillstyle=vlines,hatchsep=.5pt,hatchwidth=.2pt](1.63,-1.85){.04}
\pscircle[linewidth=.2pt,fillstyle=vlines,hatchsep=.75pt,hatchwidth=.2pt](5.81,-1.72){.07}
\pscircle[linewidth=.2pt,fillstyle=vlines,hatchsep=1pt,hatchwidth=.2pt](8.5,-1.7){.099}
\pscircle[linewidth=.2pt,fillstyle=vlines,hatchsep=2pt,hatchwidth=.2pt](12.55,-0.83){.308}
\end{picture}
\vspace{3cm}\caption{Maximal spacing of $C_r(\mathcal{X}_{1500})$ is shown in dashed lines for (a) $r=10$, (b) $r=25$, (c) $r=40$ and (d) $r=90$. The boundary of $H(\mathcal{X}_{1500})$ is shown in dotted line in (d).}\label{oooooolgiop2}
\end{figure}

Figure \ref{oooooolgiop2} shows the maximal spacings for the estimators of the Aral Sea considered in Figure \ref{oooooolgiop}. A bad choice (a big value) of the smoothing parameter allows
to detect a large gap, clearly incompatible with the uniformity hypothesis, see Figure \ref{oooooolgiop2} (d) for $r=90$. This means that the estimator contains a large spacing which is not
contained in the Aral Sea. Since the sample is uniform on the original support, we can conclude that
the smoothing parameter is too large. It must be selected smaller than $90$. The estimator of  $r_0$ is based on this idea. If we assume that the distribution is uniform on $S$, and
according to Definition \ref{sup}, $r_0$ will be estimated by
\begin{equation}\label{r0estimador}
    \hat{r}_{0}=\sup\{\gamma>0:  H_0 \mbox{ is accepted on } C_\gamma(\mathcal{X}_{n})\}.
\end{equation}

The technical aspects for the estimator defined in (\ref{r0estimador}) are considered in Sections \ref{soooooo} and  \ref{numerical}.

\section{Main results}\label{soooooo}

The existence of the supreme defined in (\ref{r0estimador}) must be guaranteed. Theorem \ref{consistencia} will show that this is the case and $\hat{r}_0$ consistently
estimates $r_0$

\begin{theorem}\label{consistencia}
Let $S\subset \mathbb{R}^{d}$ be a compact, nonconvex and nonempty
set verifying ($R_{\lambda}^r$) and $\mathcal{X}_n$ a uniform and i.i.d sample on $S$. Let $r_0$ be the parameter
defined in (\ref{maximo2}) and $\hat{r}_{0}$ defined in (\ref{r0estimador}). Let $\{\alpha_n\}\subset (0,1)$ be a sequence converging to
zero verifying $\lim_{n\rightarrow \infty}\log(\alpha_n)/n=0$. Then, $\hat{r}_0$ converges to $r_0$ in probability.
\end{theorem}
\begin{remark}
We assume that $S$ is not convex only for simplicity in the exposition. If $S$ is convex it can be shown that
$\hat{r}_0$ goes to infinity (which is the value of $r_0$ in this case)
because, with high probability, the test is not rejected for all values of $r$.
\end{remark}
%
%
%
%
%

Once the consistency of the estimator defined in (\ref{r0estimador}) has been proved, it would be
natural to study the behavior of the random set  $C_{\hat{r}_0}(\mathcal{X}_n)$ as an estimator for the support $S$. In particular, if $\lim_{r\rightarrow r_0^+}d_H(S, C_{r}(S))=0$ then
consistency of $C_{\hat{r}_0}(\mathcal{X}_n)$ can be proved easily from Theorem \ref{consistencia}. However, the
consistency can not be guaranteed if $d_H(S, C_{r}(S))$ does not go to zero as $r$ goes to $r_0$ from above (as $\hat{r}_0$ does, see below).
This problem can be solved by considering the estimator
$C_{r_n}(\mathcal{X}_n)$ where $r_n=\nu \hat{r}_0$ with $\nu\in (0,1)$ fixed. This ensures that, for $n$ large enough, with high probability
 $C_{r_n}(\mathcal{X}_n)\subset S$. In fact, Theorem \ref{consistencia2} shows
that $C_{r_n}(\mathcal{X}_n)$ achieves the same convergence rates as the convex hull of the sample for reconstructing convex sets.

\begin{theorem}\label{consistencia2}Let $S\subset \mathbb{R}^{d}$ be a compact, nonconvex and nonempty set
verifying ($R_{\lambda}^r$) and $\mathcal{X}_n$ a uniform and i.i.d sample on $S$. Let $r_0$ be the parameter
defined in the (\ref{maximo2}) and $\hat{r}_{0}$ defined in (\ref{r0estimador}). Let $\{\alpha_n\}\subset (0,1)$ be a sequence converging to zero under
the conditions of Theorem \ref{consistencia}. Let be $\nu \in (0,1)$ and $r_n=\nu \hat{r}_0$. Then,
$$
   d_H(S, C_{r_n}(\mathcal{X}_n))=O_P\left(\frac{\log n}{n}\right)^{\frac{2}{d+1}}.
$$
The same convergence order holds for $d_H(\partial S, \partial C_{r_n}(\mathcal{X}_n))$ and $d_\mu(S\triangle C_{r_n}(\mathcal{X}_n))$.
\end{theorem}

\begin{remark}
The selector proposed by Mandal and Murthy (1997), $r_n^{MM}$, goes to zero in probability. In Pateiro-L\'opez
and Rodr\'{\i}guez-Casal (2013) it is proved that, for a deterministic sequence of
parameters $d_n$ ($d_n\leq r_0$ and $d_n^2 n/\log (n)\to \infty$), the convergence rate (in probability) for the distance in measure
is, for the bidimensional case, $d_n^{-1/3}n^{-2/3}$.
This is the convergence rate of the new proposal plus a penalizing term $d_n^{-1/3}$ which goes to infinity if $d_n\to 0$. It is expected that this penalizing
factor, $(r_n^{MM})^{-1/3}$ also appears
for the the Mandal and Murthy's proposal.
\end{remark}

\section{Numerical aspects of the algorithm}\label{numerical}

The practical implementation of this method requires considering some numerical aspects in order to detail it completely.

With probability one, for $n$ large enough, the existence of the estimator defined in (\ref{r0estimador}) is guaranteed under the hypotheses of
Theorem \ref{consistencia}. However, in practise, this estimator might not exist for a specific sample $\mathcal{X}_n$ and a given value of
the significance level $\alpha$. Therefore, the influence of $\alpha$ must be taken into account. The null hypothesis
will be (incorrectly) rejected  on $C_r(\mathcal{X}_n)$ for $0< r\leq r_0$ with
probability $\alpha$ approximately. This is not important from the theoretical point of view, since we are assuming that $\alpha=\alpha_n$ goes to zero as
the sample size increases. But, what to do if, for a given sample, we reject $H_0$ for {\it all} $r$ (or at least {\it all} reasonable values of $r$)? In order to fix a minimum acceptable
value of $r$, it is assumed that $S$ (and, hence, the estimator) will
have no more than $C$ cycles. Too split estimators will not be considered even in the case that we reject $H_0$ for all $r$. The minimum value
that ensures a number of cycles not greater than $C$ will be taken in this latter case, see below.

Dichotomy algorithms can be used to compute $\hat{r}_0$. The practitioner
must select a maximum number of iterations $I$ and two initial points $r_m$ and $r_M$ with $r_m<r_M$ such that the null hypothesis of uniformity
is rejected and accepted on $C_{r_M}(\mathcal{X}_{n})$ and $C_{r_m}(\mathcal{X}_{n})$, respectively. According to the
previous comments, it is assumed that the number of cycles of $C_{r_m}(\mathcal{X}_n)$ must not be greater than $C$. Choosing a value close
enough to zero is usually sufficient to select $r_m$. However, if selecting this $r_m$ is not possible because, for very low and positive values of $r$, the hypothesis
of uniformity is still rejected on $C_{r}(\mathcal{X}_{n})$ then $r_0$ is estimated as the positive closest value to zero $r$ such
that the number of cycles of $C_r(\mathcal{X}_n)$ is smaller than or equal to $C$. On the other hand, if the hypothesis of uniformity
is accepted even on $H(\mathcal{X}_{n})$ then we propose $H(\mathcal{X}_n)$ as the estimator for the support.

To sum up, the next inputs should be given: the significance level $\alpha\in (0,1)$, a maximum number of iterations $I$, a maximum number of cycles $C$ and two initial
values $r_m$ and $r_M$. Given these parameters $\hat{r}_0$ will be computed as follows:  \vspace{.08cm}

\begin{enumerate}
\item In each iteration and while the number of them is smaller than $I$:\vspace{.08cm}
    \begin{enumerate}
    \item $r=(r_m+r_M)/2.$\vspace{.08cm}
    \item If the null hypothesis is not rejected on $C_{r}(\mathcal{X}_n)$ then $r_m=r$.\vspace{.08cm}
    \item Otherwise, $r_M=r$.\vspace{.08cm}
    \end{enumerate}
\item Then, $\hat{r}_0=r_m$.\vspace{.08cm}
\end{enumerate}


According to the correction of the bias proposed by Ripley and
Rasson (1977) for the convex hull estimator, Berrendero et al. (2012) suggested rejecting
the null hypothesis of uniformity when
$$
    \hat{V}_{n,r}>\frac{\mu(S_n)(u_\alpha+\log{n}+(d-1)\log{\log{n}}+\log{\beta})}{n-v_n},
$$
where  $v_n$ denotes the number of vertices of $S_n=C_{r}(\mathcal{X}_n)$ (points of $\mathcal{X}_n$ that belong to $\partial S_n$). In this work, it is
proposed to redefine the critical region as
$$
    \hat{V}_{n,r}>\hat{c}_{n,\alpha,r}^*,
$$where $\hat{c}_{n,\alpha,r}^*$ is equal to
$$
\small{\frac{\mu(S_n)(u_\alpha+\log{(n-v_n)}+(d-1)\log{\log{(n-v_n)}}+\log{\beta})}{n-v_n}},
$$
that is, we suggest to replace $n$ by $n-v_n$ in the definition of $\hat{c}_{n,\alpha,r}$ elsewhere not only in the denominator.
Although the main theoretical results in Section \ref{soooooo} are established in terms of $\hat{c}_{n,\alpha,r}$ instead of $\hat{c}_{n,\alpha,r}^*$, the proofs
are completely analogous in both cases since $v_n$ is negligible with respect to $n$ see, for instance, the
upper bound for the expected number of vertices in Theorem 3 by Pateiro-L\'opez and Rodr\'iguez-Casal (2013).

Some technical aspects related to the computation of the
maximal spacings must be also considered. Testing the null hypothesis of uniformity is a procedure repeated $I$ times in this algorithm. This may seem to be very computing intensive
since the test involves calculating the maximal spacing.
However, we do not need to know the exact value of the maximal spacing since we are not interested in computing
the test statistic. 
In fact, it is only necessary to check if, for a fixed $r$,  $C_{r}(\mathcal{X}_n)$ contains an open ball, that does not intersect the sample points
with volume greater than the test's critical value $\hat{c}_{n,\alpha,r}^*$. In other words, we will simply
check if an open ball of radius equal to $\hat{c}_{n,\alpha,r}^*$ and center $x$ is contained in $C_{r}(\mathcal{X}_n)\setminus \mathcal{X}_n$. If this disc exists
then $x\notin B_{\hat{c}_{n,\alpha,r}^*}(\mathcal{X}_n)$ where\vspace{-.05cm}
$$B_{\hat{c}_{n,\alpha,r}^*}(\mathcal{X}_n)=\bigcup_{X_i \in \mathcal{X}_n}B_{\hat{c}_{n,\alpha,r}^*}(X_i)\vspace{-.05cm}$$
is the dilation
of radius $\hat{c}_{n,\alpha,r}^*$ of the sample. Therefore, the centers of the possible maximal balls necessarily
lie outside $B_{\hat{c}_{n,\alpha,r}^*}(\mathcal{X}_n)$. Following Berrendero et al. (2012), to check
if the null hypothesis of uniformity is rejected on $C_{r}(\mathcal{X}_n)$, we will follow the next steps:\vspace{.08cm}
\begin{enumerate}
  \item Determine the set $D(r) = C_{r}(\mathcal{X}_n)\cap \partial B_{\hat{c}_{n,\alpha,r}^*}(\mathcal{X}_n)$. Notice that,
  if $x\in D(r)$ then $B_{\hat{c}_{n,\alpha,r}^*}(x)\cap \mathcal{X}_n=\emptyset$.\vspace{.08cm}
\item Calculate $M(r)=\max\{d(x,\partial C_{r}(\mathcal{X}_n):x\in D(r)\}$.\vspace{.08cm}
  \item If $M(r)\leq \hat{c}_{n,\alpha,r}^*$ then the null hypothesis of uniformity is not rejected.\vspace{.08cm}
\end{enumerate}

It should be noted that $\partial C_{r}(\mathcal{X}_n)$ and $\partial B_{\hat{c}_{n,\alpha,r}^*}(\mathcal{X}_n)$ can be easily computed (at least for the bidimensional case), see Pateiro-L\'opez and Rodr\'iguez-Casal (2010). \vspace{-.15cm}

\section{Simulation study}\label{simulation}

The performances of the algorithm proposed in this paper and Mandal and Murthy (1997)'s method will be analyzed in this section. They will be denoted by RS and MM, respectively. A total of $1000$ uniform samples of four different sizes $n$ have been generated on three support models in the Euclidean space $\mathbb{R}^2$, see Figure \ref{oooofffffdfffg}.

The first set, $S=B_{0.35}[(0.5,0.5)]\setminus B_{0.15}((0.5,0.5))$, is a circular ring with $r_0=0.15$. The other two ones are
two interesting sets, $S=\textbf{\large{C}}$ and $S=\textbf{\large{S}}$ with $r_0=0.2$ and $r_0=0.0353$, respectively. The values
of $n$ considered are $n=100$, $n=500$, $n=1000$ and $n=1500$. In addition, four values
for $\alpha$ have been taken into account, $\alpha_1=10^{-1}$, $\alpha_2=10^{-2}$, $\alpha_3=10^{-3}$ and $\alpha_4=10^{-4}$. The maximum number of cycles $C$ was fixed equal to $4$.

\begin{figure}
\hspace{-1.9cm}\begin{picture}(-160,150)
\put(29,30){\includegraphics[width=5.1 cm,height=5.3cm]{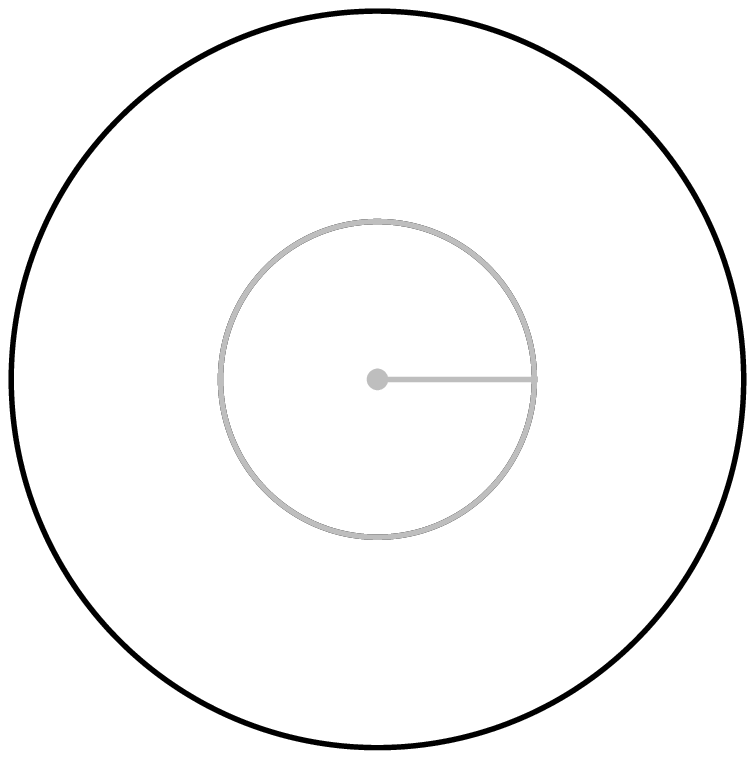}}
\put(313,30){\includegraphics[width=5.1 cm,height=5.3cm]{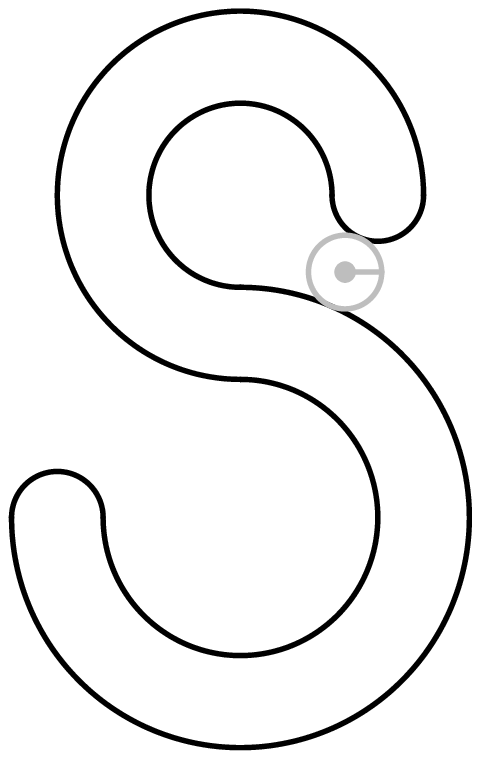}}
\put(171,30){\includegraphics[width=5.1 cm,height=5.3cm]{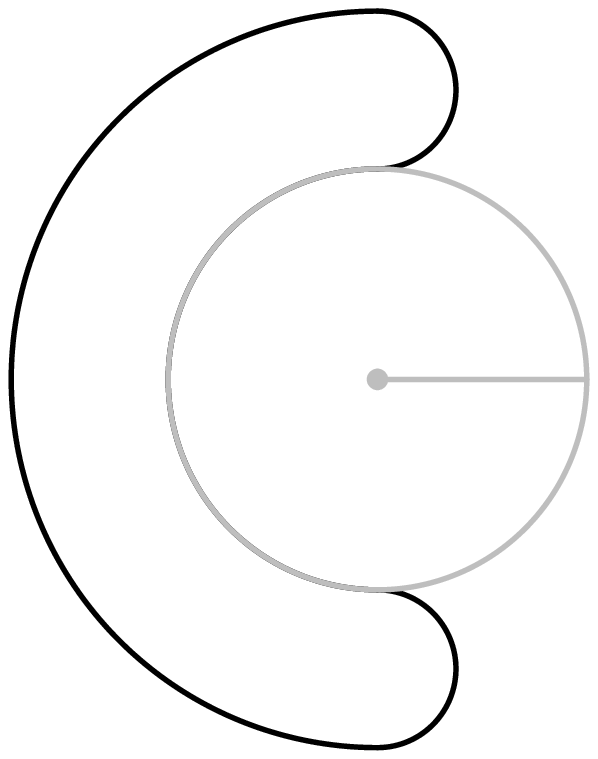}}
\put(29,-90){\includegraphics[width=5.1 cm,height=5.3cm]{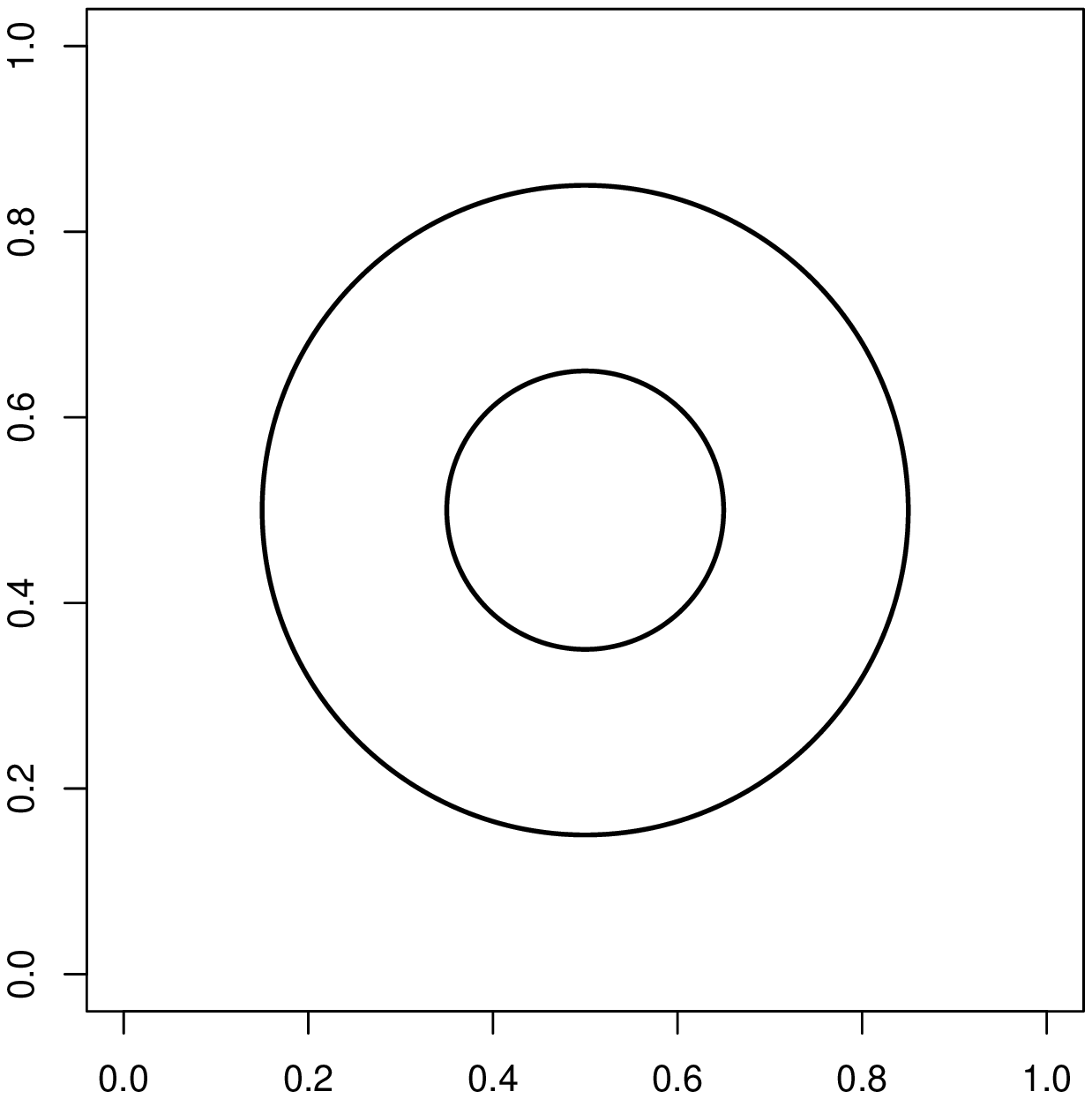}}
\put(313,-90){\includegraphics[width=5.1 cm,height=5.3cm]{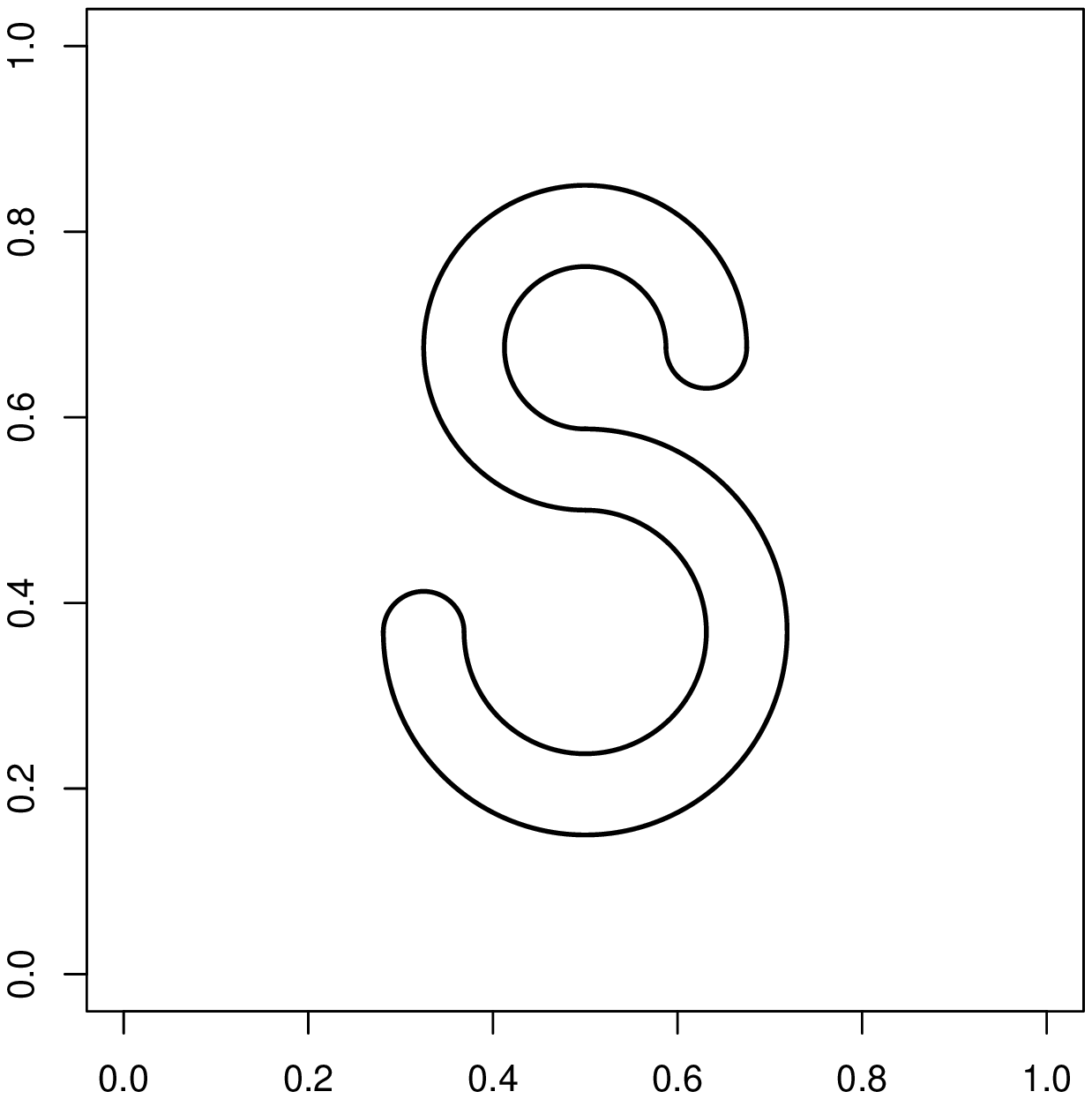}}
\put(171,-90){\includegraphics[width=5.1 cm,height=5.3cm]{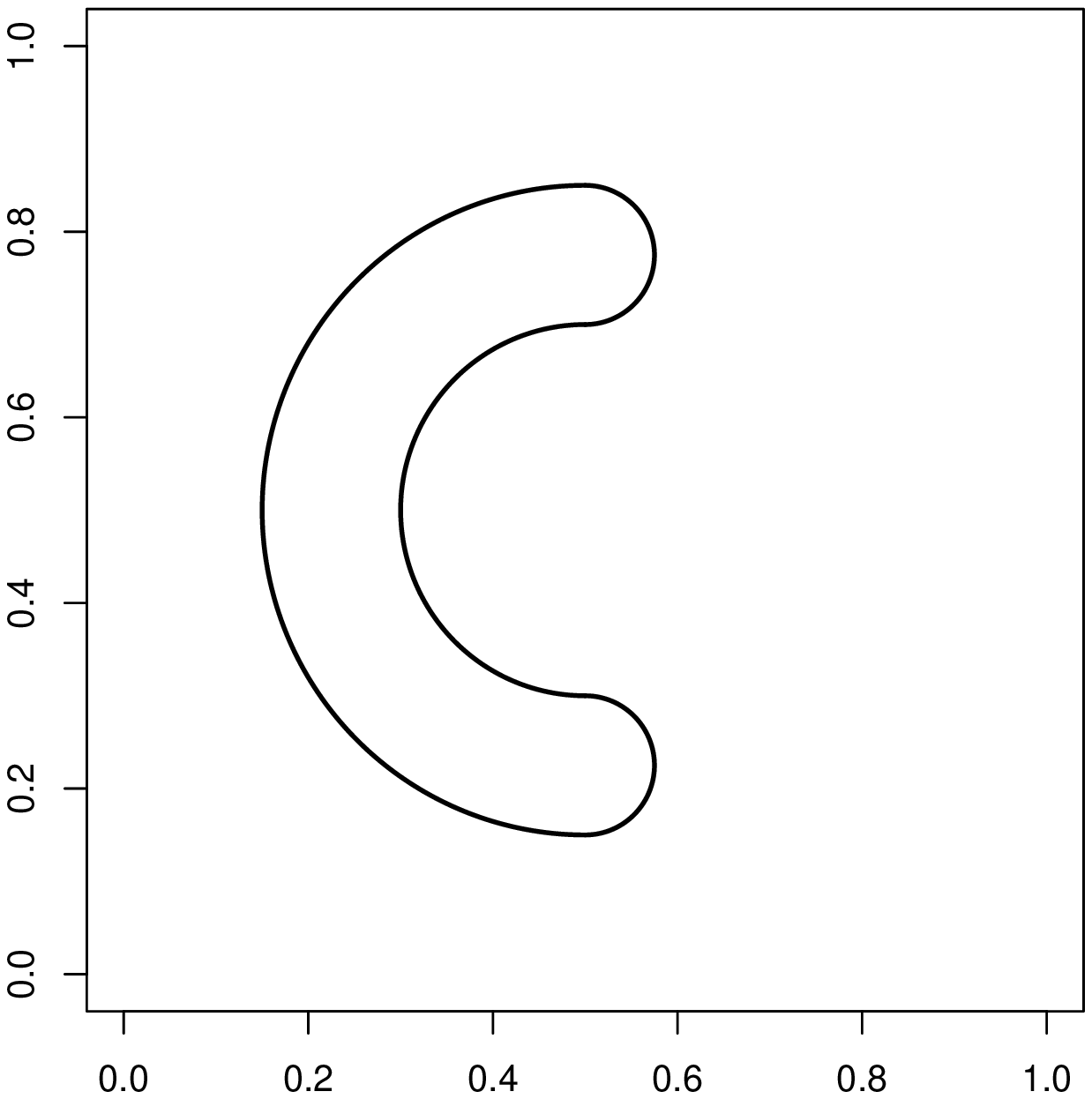}}
\put(241,167){(b)}
\put(101,167){(a)}
\put(383,167){(c)}
\end{picture}
$$ $$
$$ $$
$$ $$
$$\vspace{.5mm} $$
\caption{(a) $S=B_{0.35}[(0.5,0.5)]\setminus B_{0.15}((0.5,0.5))$ with $r_0=0.15$. (b) $S=\textbf{\large{C}}$ with $r_0=0.2$. (c) $S=\textbf{\large{S}}$ with $r_0=0.0353$.}\label{oooofffffdfffg}
\end{figure}

For each fixed random sample, both estimators of the smoothing parameter of the $r-$convex hull have
been calculated. So, one thousand estimations have been obtained for each algorithm, fixed a model and the values of $n$ and $\alpha$. The empirical means
of these one thousand estimations are showed in Tables \ref{1111111b}, \ref{22222222b} and \ref{333333333b} for the RS and MM methods. We should mention that MM method is included in these table only for illustrative purposes. The results of these two algorithms are not directly comparable since the goal of MM is not to estimate the parameter $r_0$ defined in (\ref{maximo2}). However, comparing the behavior of the two resulting support estimators can be really interesting. Tables \ref{333333333c}, \ref{333333333cc} and \ref{333333333ccc} contain the empirical means of one thousand Monte Carlo estimations for the distance in measure between the RS and MM support estimators and the corresponding theoretical models, respectively. In addition, we have estimated the distance in measure between the $r_0-$convex hull of each sample and its corresponding support model for the different sample sizes. The means of these estimations can be considered as a reference value. They are showed (multiplied by $10$) in the last row of Tables \ref{333333333c}, \ref{333333333cc} and \ref{333333333ccc}. A grid of $334^2$ points was considered in the unit square for estimating the distance in measure. The parameter $\nu$ was fixed equal to $0.95$ for calculating the RS support estimator.

\begin{table}[h!]
\caption{Empirical means of $1000$ RS and MM estimations for the smoothing parameter of the $r-$convex hull with $S=B_{0.35}[(0.5,0.5)]\setminus B_{0.15}((0.5,0.5))$. In this case, $r_0=0.15$.}\label{1111111b}
 \begin{center}
\begin{tabular}{ccrrrr}
    &  $n$& $100$&$500$&$1000$ &$1500$\\
        \hline
     \multirow{4}{*}{RS}& & & & & \\
      & $\alpha_1=10^{-1}$&  0.1592  &  0.1456 &  0.1438   & 0.1410    \\
      & $\alpha_2=10^{-2}$& 0.1592   &  0.1509 & 0.1499   &  0.1495 \\
      & $\alpha_3=10^{-3}$& 0.1592    & 0.1516 & 0.1507   & 0.1503 \\
      & $\alpha_4=10^{-4}$&  0.1592  &  0.1517   &  0.1507    &  0.1504  \\
     MM& &   0.1969  &  0.1295 &   0.1084 &  0.0977  \\
    \hline
\end{tabular}
 \end{center}
  \end{table}

  \begin{table}[h!]
  \caption{Empirical means of $1000$ RS and MM estimations for the smoothing parameter of the $r-$convex hull with $S=\textbf{\large{C}}$. In this case, $r_0=0.2$.}\label{22222222b}
 \begin{center}
\begin{tabular}{ccrrrr}
    &  $n$& $100$&$500$&$1000$ &$1500$\\
        \hline
     \multirow{4}{*}{RS}& & & & & \\
      & $\alpha_1=10^{-1}$&  0.2724   &   0.2007   & 0.1903   & 0.1888  \\
      & $\alpha_2=10^{-2}$&  0.2929 &  0.2150   &  0.2056  & 0.2032    \\
      & $\alpha_3=10^{-3}$&   0.2982   &  0.2188   & 0.2089    &  0.2055  \\
      & $\alpha_4=10^{-4}$&  0.2988  &  0.2226  &   0.2105 &  0.2068 \\
     MM& &    0.1636  &   0.1072  &  0.0897   &   0.0809 \\
    \hline
\end{tabular}
 \end{center}\end{table}

 \begin{table}[h!]
 \caption{Empirical means of $1000$ RS and MM estimations for the smoothing parameter of the $r-$convex hull with $S=\textbf{\large{S}}$. In this case, $r_0= 0.0353$.}\label{333333333b}
 \begin{center}
\begin{tabular}{ccrrrr}
   &  $n$& $100$&$500$&$1000$ &$1500$\\
        \hline
     \multirow{4}{*}{RS}& & & & & \\
      & $\alpha_1=10^{-1}$& 0.0954 &   0.0833  & 0.0637  &  0.0548   \\
      & $\alpha_2=10^{-2}$&  0.0954   &  0.0878    &  0.0695  & 0.0602   \\
      & $\alpha_3=10^{-3}$& 0.0958  &  0.0886  &  0.0736  &  0.0631  \\
      & $\alpha_4=10^{-4}$& 0.1077  & 0.0887 & 0.0778  &  0.0659   \\
     MM& &   0.1644 & 0.1055 & 0.088 &  0.0792  \\
    \hline
\end{tabular}
 \end{center}\end{table}

\newpage

 \begin{table}[h!]
  \caption{Empirical means of $1000$ estimations (multiplied by $10$) obtained for the distance in measure between $S=B_{0.35}[(0.5,0.5)]\setminus B_{0.15}((0.5,0.5))$ and the resulting support estimators for RS and MM methods. The last row contains the benchmarks (multiplied by $10$) for each sample size.}\label{333333333c}
 \begin{center}
\begin{tabular}{ccrrrr}
    &  $n$& $100$&$500$&$1000$ &$1500$\\
        \hline
     \multirow{4}{*}{RS}& & & & & \\
      & $\alpha_1=10^{-1}$&    0.9288   & 0.3293  &  0.2085    &       0.1623  \\
      & $\alpha_2=10^{-2}$&  0.9288   &  0.3143   & 0.1970    &   0.1492    \\
      & $\alpha_3=10^{-3}$& 0.9294     &    0.3123  & 0.1957  &   0.1484  \\
      & $\alpha_4=10^{-4}$&  0.9288    &   0.3122   &   0.1957  &   0.1483    \\
     MM& &    1.4165   &     0.3378  &   0.2316   &  0.1837  \\
    \hline
        & &   0.9337  & 0.2956 &0.1819&  0.1364 \\
      \end{tabular}
 \end{center}
 \end{table}

 $ $

 $ $

 $ $

 $ $

 $ $

 $ $

 \begin{table}[h!]
 \caption{Empirical means of $1000$ estimations (multiplied by $10$) obtained for the distance in measure between $S=\textbf{\large{C}}$ and the resulting support estimators for RS and MM methods. The last row contains the benchmarks (multiplied by $10$) for each sample size.}\label{333333333cc}
 \begin{center}
\begin{tabular}{ccrrrr}
    &  $n$& $100$&$500$&$1000$ &$1500$\\
        \hline
     \multirow{4}{*}{RS}& & & & & \\
      & $\alpha_1=10^{-1}$&  0.6041   &  0.1472  &    0.0920  &   0.0712    \\
      & $\alpha_2=10^{-2}$&   0.6677    &   0.1589 &  0.0833 &  0.0640   \\
      & $\alpha_3=10^{-3}$&   0.6820  &    0.1953  &   0.0832  & 0.0631    \\
      & $\alpha_4=10^{-4}$&   0.6837  &  0.2440  &    0.0865  &  0.0626    \\
     MM& &     0.4145   &    0.1681  &    0.1125 &  0.0885  \\
    \hline
        & &    0.3727  & 0.1277 &  0.0800&  0.0606 \\
\end{tabular}
 \end{center}
 \end{table}

 $ $

 \newpage

 \begin{table}[h!]
  \caption{Empirical means of $1000$ estimations (multiplied by $10$) obtained for the distance in measure between $S=\textbf{\large{S}}$ and the resulting support estimators for RS and MM methods. The last row contains the benchmarks (multiplied by $10$) for each sample size.}\label{333333333ccc}
 \begin{center}
\begin{tabular}{ccrrrr}
    &  $n$& $100$&$500$&$1000$ &$1500$\\
        \hline
     \multirow{4}{*}{RS}& & & & & \\
      & $\alpha_1=10^{-1}$&0.6389    &  0.2591     &   0.1842  &  0.1485  \\
      & $\alpha_2=10^{-2}$& 0.6389     & 0.2537 &    0.1821  &  0.1455    \\
      & $\alpha_3=10^{-3}$& 0.6411  & 0.2530     &   0.1821  & 0.1464   \\
      & $\alpha_4=10^{-4}$&0.6797     & 0.2529   & 0.1816   &    0.1476    \\
     MM& &  1.2319  &  0.4851 &  0.2445 & 0.1514 \\
    \hline
        & &   1.0794  &  0.3320 &   0.2038&   0.1541  \\
\end{tabular}
 \end{center}\end{table}

$ $

$ $

\begin{figure}[h!]\hspace{-1.9cm}
\begin{picture}(10,150)
\put(30,20){\includegraphics[width=5cm,height=5cm]{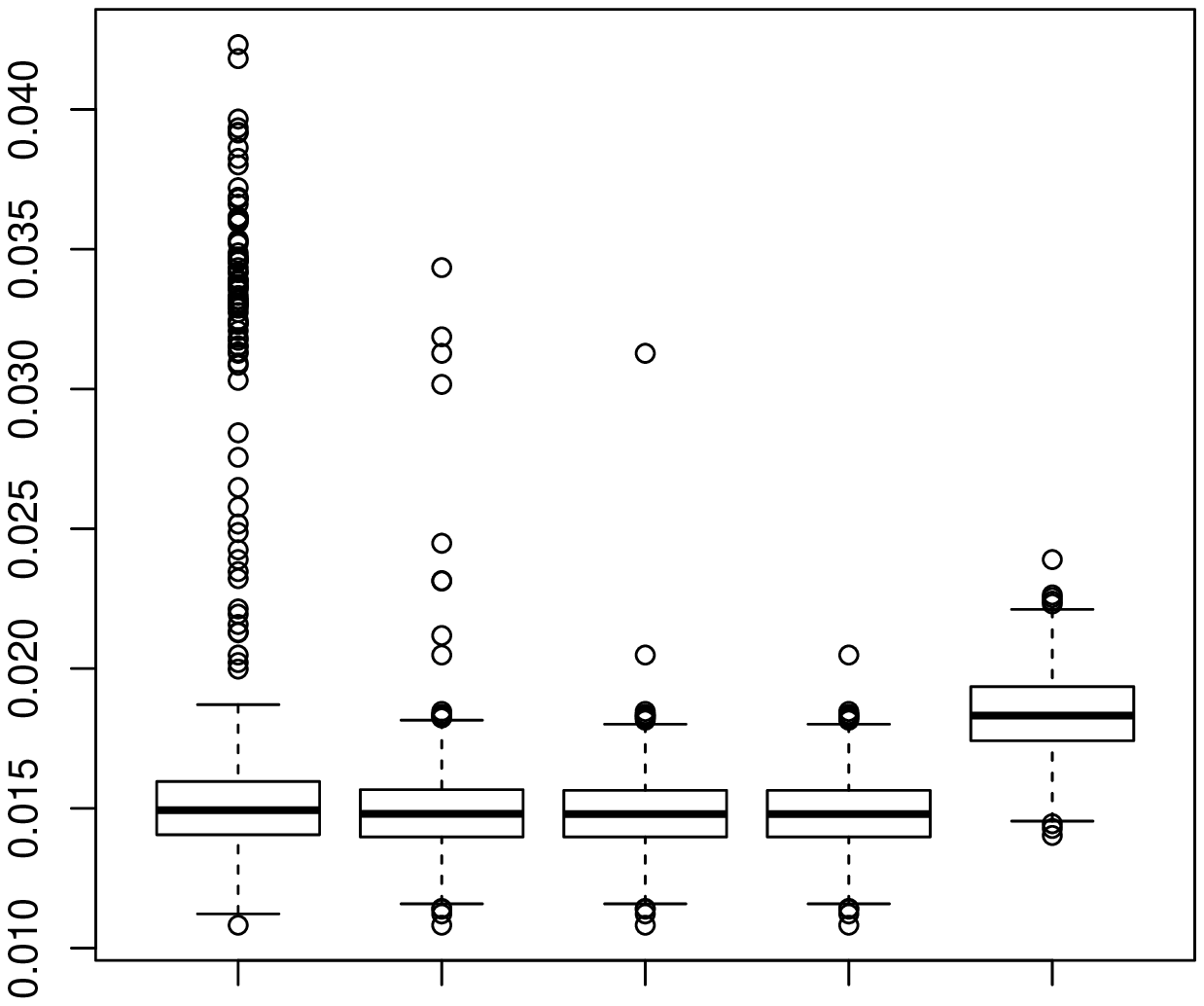}}
\put(172,20){\includegraphics[width=5cm,height=5cm]{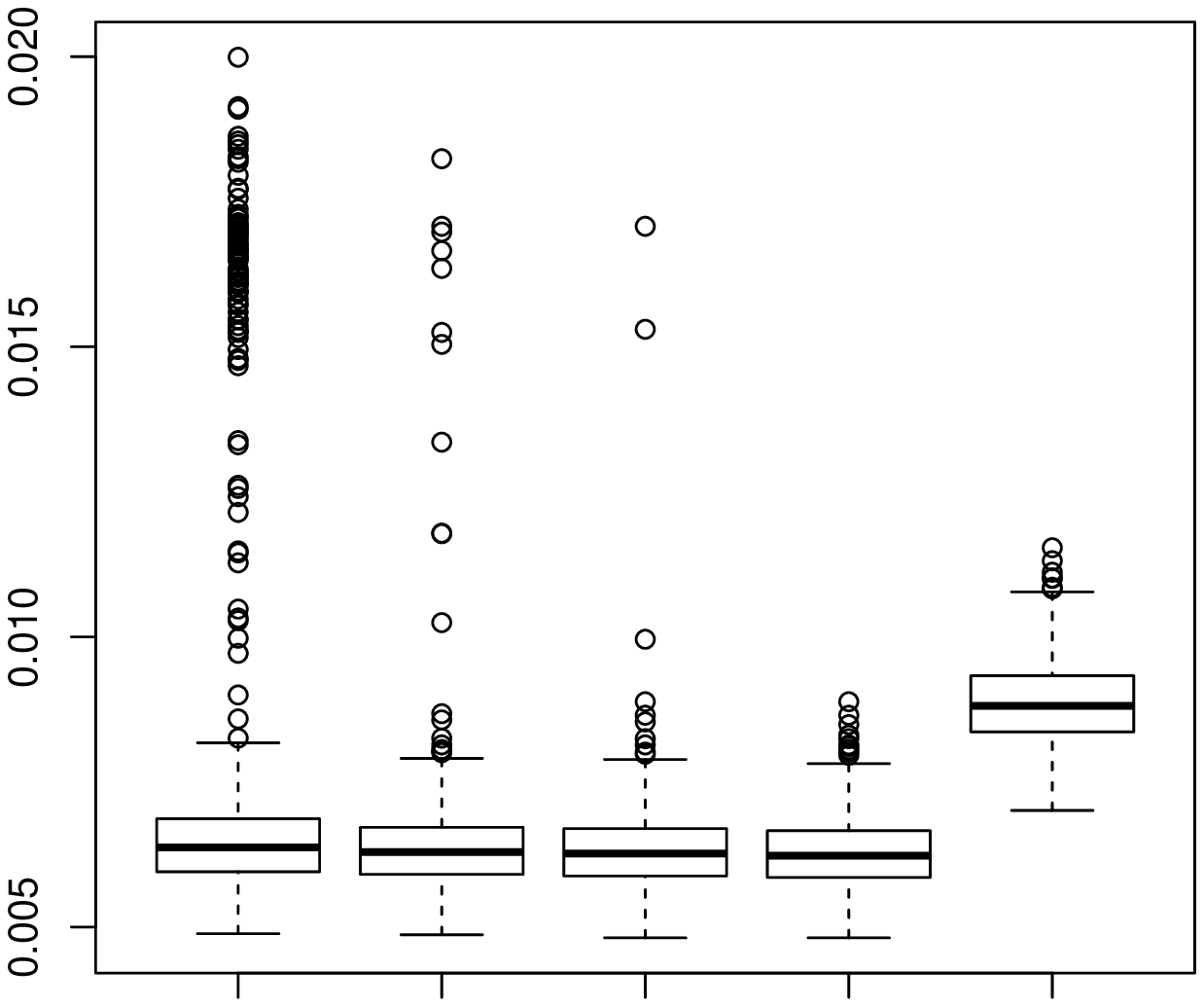}}
\put(314,20){\includegraphics[width=5cm,height=5cm]{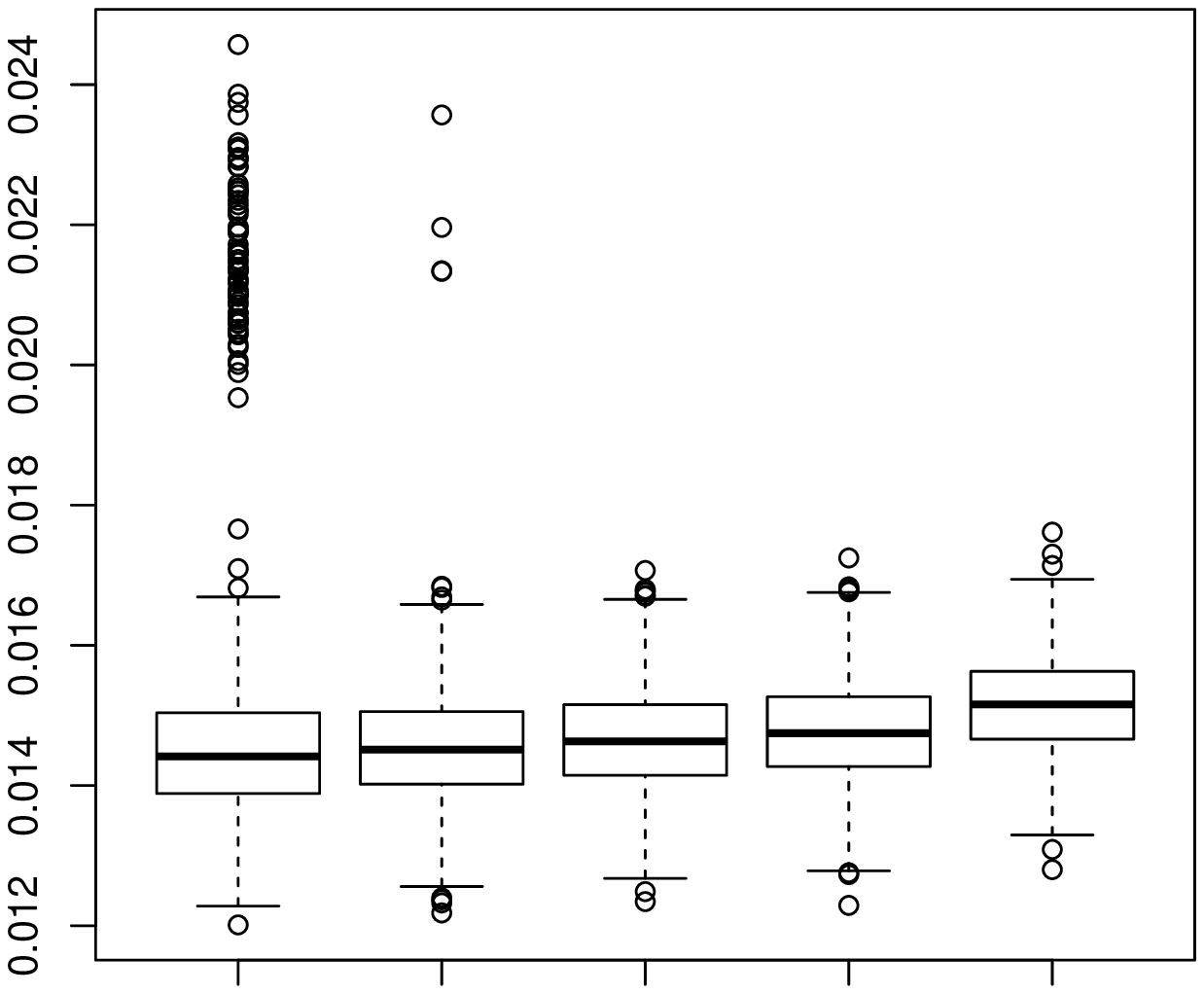}}
\put(100,160){(a)}
\put(242,160){(b)}
\put(385,160){(c)}
\end{picture}\vspace{-1. cm}
\caption{Boxplots of the estimations for the distance in measure for RS and MM methods when $n=1500$ for (a) $S=B_{0.35}[(0.5,0.5)]\setminus B_{0.15}((0.5,0.5))$, (b) $S=\textbf{\large{C}}$ and (c) $S=\textbf{\large{S}}$. From left to right, RS considering $\alpha_1$, $\alpha_2$, $\alpha_3$ and $\alpha_4$ and MM.}\label{oooofffffffuuu1}
\end{figure}

Figure \ref{oooofffffffuuu1} contains the boxplots for the estimations of the distance in measure between the resulting support estimators for the RS and MM methods when $n=1500$.

\emph{Conclusions.} According to the results showed in Tables \ref{1111111b}, \ref{22222222b} and \ref{333333333b}, RS presents a good global behavior for estimating the smoothing parameter of the $r-$convex hull. Only when $S=\textbf{\large{C}}$ and $n=100$, MM provides better results, see Table \ref{22222222b}. In this particular case, the estimations of RS are specially greater than $0.2$, the real value of parameter $r_0$. In general, MM provides too small estimations, mainly for high values of the sample size, see Tables \ref{1111111b} and \ref{22222222b}.\\$\vspace{-1.05cm} $\\

The role of the level of significance $\alpha$ must be also discussed. Taking low values of $\alpha$ reduces the number of outliers considerably for the three support models presented. In addition, if the model considered is not too complex then small
values of $\alpha$ provide better results for $n$ large enough reducing the risk of rejecting
the null hypothesis of uniformity when it is satisfied, see for instance $S=B_{0.35}[(0.5,0.5)]\setminus B_{0.15}((0.5,0.5))$ or $S=\textbf{\large{C}}$ in Tables \ref{1111111b} and \ref{22222222b}. Therefore, excessively low values of $r$ will not be selected. However, if the support model is not so simple then choosing large values of $\alpha$ provides better estimations for the smoothing parameter, see Table \ref{333333333b} for $S=\textbf{\large{S}}$. Anyway, for moderate and large values of the sample size the dependence on $\alpha$ of RS method is small.\\$\vspace{-.8cm} $\\

Finally and according to Tables \ref{333333333c}, \ref{333333333cc} and \ref{333333333ccc}, RS always provides the smallest estimation errors for the criteria considered except when $S=\textbf{\large{C}}$ with $n=100$ or even $n=500$ if the value of $\alpha$ is too large, see Table \ref{333333333cc}. Therefore, RS support estimator is more competitive than MM algorithm. According to the previous comments, it can be seen that the number of outliers for RS increases if large values of $\alpha$ are considered for the three support models, see Figure \ref{oooofffffffuuu1}.

\section{Proofs}\label{prprpr}In this section the proofs of the stated theorems are presented.\vspace{.15cm}\\
\emph{Proof of Proposition \ref{nonece}}\vspace{.15cm}\\
Let us suppose that $r_n\neq \overline{r}$ for all $n$ since that, otherwise, the proof would be trivial. It is verified that

$$\forall a\in\partial A\mbox{ and }\forall n\in \mathbb{N}\mbox{ }\exists x_n \mbox{ such that }a \in B_{r_n}[x_n]\mbox{ and } B_{r_n}(x_n)\cap A=\emptyset.\vspace{1.9mm}$$For each $a\in \partial A$, let us consider the sequence of closed balls $\{B_{r_n}(x_n)\}$. It is not restrictive to assume that $\{r_n\}$ is a monotone increasing sequence. In another case, it would be possible to consider a monotone subsequence of $\{r_n\}$ denoted by $\{r_n\}$ again converging to $\overline{r}$. If a decreasing subsequence was considered, the proof would be trivial. Then, only the increasing case will be considered. Then, $\{r_n\}$ converges to $\overline{r}$ and $\{x_n\}$ converges to $x_a$ since $\{x_n\}$ is bounded and it contains a convergent
subsequence which we denote by $\{x_n\}$ again. Two steps are necessary to get the proof.\vspace{3mm}\\ \underline{Step 1:} It will be proved that for any $a\in\partial A$ it is verified that $B_{\overline{r}}(x_a)\cap A=\emptyset$.
To see this suppose the contrary, that is, let us suppose that there exists $a\in \partial A$ such that $B_{\overline{r}}(x_a)\cap A\neq \emptyset$ . Then, there exists $ \overline{a}\in B_{\overline{r}}(x_a)\cap A$ verifying that $\|\overline{a}-x_a\|<\overline{r}$. \\Since $\{r_n\}\uparrow \overline{r}$,
$$
   \exists n_0 \in \mathbb{N}\mbox{ such that }\|\overline{a}-x_a\|<r_n<\overline{r},\mbox{ }\forall n \geq n_0.\vspace{1.5mm}
$$
So,
\begin{equation}\label{2n}
   \forall n\geq n_0,\mbox{ }\overline{a}\in B_{r_n}(x_a) .\vspace{1.5mm}
\end{equation}Let us define for all $ n\geq n_0$,
$$d_n=d(\overline{a},\partial B_{r_n}(x_a))\vspace{1.5mm}.$$In addition, $\{r_n\}$ is an increasing sequence. Then, it is verified that
$$B_{r_1}(x_a)\subset B_{r_2}(x_a)\subset...\vspace{1.5mm}$$and, as consequence and taking (\ref{2n}) into account,

$$0<d_{n_0}\leq d_{n_1}\leq d_{n_2}\leq ...\vspace{1.5mm}$$Let us consider $d_{n_0}/2$, since $\{x_n\}$ converges to $x_a$,
$$\exists n_1\in \mathbb{N}\mbox{ such that }\|x_a-x_n\|<d_{n_0}/2,\mbox{ }\forall n \geq n_1.$$So,
$$\overline{a}\in B_{r_n}(x_n),\mbox{ }\forall n \geq n_2=\max\{n_0,n_1\}.$$To see this, notice that, if $n\geq n_2$ then
$$\|\overline{a}-x_n\|\leq\|\overline{a}-x_a\|+\|x_a-x_n\|<r_n-d_n+\frac{d_{n_0}}{2}<r_n-d_n+\frac{d_n}{2}<r_n .\vspace{1.5mm}$$This fact is a contradiction since $B_{r_n}(x_n)\cap A =\emptyset$, for all $n$ and $\overline{a}\in A$.\vspace{1.6mm}\vspace{3mm}\\\underline{Step 2:} It will be proved that $a\in B_{\overline{r}}[x_a]$. We will assume that $a\notin B_{\overline{r}}[x_a]$ and we will show that this is impossible under the assumptions we have done. If $a\notin B_{\overline{r}}[x_a]$ then $\|a-x_a\|>\overline{r}$ and it is possible to define $\epsilon=\|a-x_a\|-\overline{r}>0$. Since $\{x_n\}$ converges to $x_a$, there exists $n_0\in \mathbb{N}$ such that $\|x_n-x_a\|<\epsilon$. For all $n\geq n_0$,
 $$\|x_n-a\|\geq \|a-x_a\|-\|x_a-x_n\|>\|a-x_a\|-\epsilon=\overline{r} .$$
Since $\{r_n\}$ is an monotone increasing sequence converging to $\overline{r}$, $a\notin B_{r_n}[x_n]$. This is a contradiction since we are assuming that  $a\in B_{r_n}[x_n]$ for all $n$.\hfill $\Box$\vspace{.15cm}\\
\emph{Proof of Proposition \ref{rconvexo}}\vspace{.15cm}\\Some auxiliary results are necessary. Lemma \ref{a} guarantees the existence of a unit vector for each point belonging to the boundary of the set.
\begin{lemma}\label{a}Let $A\subset \mathbb{R}^{d}$ be a closed and nonempty set such that $A^c$ satisfies the $\lambda-$rolling property. Then, for all $a\in\partial A$ exists $\eta(a)\label{eta2}$ (not necessarily unique) such that $\|\eta(a)\|=1$ and $B_{\lambda}[a-\lambda\eta(a)]\subset A.\vspace{1.4mm}$
\end{lemma}

\begin{proof}According to the property of rolling freely for a given $a\in\partial A^c=\partial A$ exists $B_\lambda[x]$ such that $a\in B_\lambda[x]$ and $B_\lambda(x)\cap A^c=\emptyset$. Therefore, $B_\lambda[x]\subset A$ and it is verified that $\|x-a\|=\lambda$. If $\|x-a\|<\lambda$ then $a\in B_{\lambda-\|x-a\|}[x]\subset \interior(A) $ which is a contradiction since that $a\in\partial A$. Then, it is possible to define $\eta(a)=(a-x)/\|a-x\|$. It is verified that $x=a-\lambda\eta(a)$. So, $B_{\lambda}[a-\lambda\eta(a)]\subset A$.\qedhere
\end{proof}

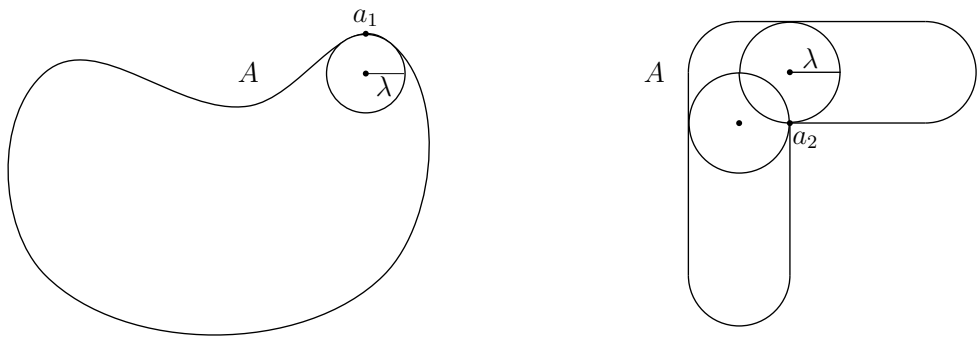
\begin{figure}[h!]\centering
\begin{pspicture}(-.1,0)(14,4.5)
 \scalebox{.9}[0.9]{\psccurve[showpoints=false,fillstyle=solid,fillcolor=white,linecolor=black,linewidth=0.2mm,linearc=3](1,1)(1,4)(4,3.5)(6,4.5)(6,1)
\pscircle[fillstyle=solid,fillcolor=white,linecolor=black,linewidth=0.2mm,linearc=3](5.735,3.98){.59}
\psdots*[dotsize=2.5pt](5.735,3.98)
\psline[linearc=0.25,linecolor=black,linewidth=0.1mm](5.735,3.98)(6.3,3.98)
\rput(6,3.8){$\tiny{\lambda}$}
\psdots*[dotsize=2.5pt](5.736,4.57)
\rput(5.736,4.8){$\tiny{a_1}$}
\psline[linearc=0.25,linecolor=black,linewidth=0.2mm](12,1)(12,3.25)
\psline[linearc=0.25,linecolor=black,linewidth=0.2mm](12,3.25)(14,3.25)
\psarc[showpoints=false,linearc=0.25,linecolor=black,linewidth=0.2mm](14,4.){.75}{-90}{90}
\psarc[showpoints=false,linearc=0.25,linecolor=black,linewidth=0.2mm](11.25,1){.75}{-180}{0}
\psarc[showpoints=false,linearc=0.25,linecolor=black,linewidth=0.2mm](11.25,4){.75}{90}{180}
\psline[linearc=0.25,linecolor=black,linewidth=0.2mm](10.5,1)(10.5,4)
\psline[linearc=0.25,linecolor=black,linewidth=0.2mm](14,4.75)(11.25,4.75)
\pscircle[fillstyle=solid,fillcolor=white,linecolor=black,linewidth=0.2mm,linearc=3](12,4){.75}
\pscircle[fillstyle=solid,fillcolor=white,linecolor=black,linewidth=0.2mm,linearc=3](11.25,3.25){.75}
\psarc[showpoints=false,linearc=0.25,linecolor=black,linewidth=0.2mm](12,4){.75}{180}{270}
\psdots*[dotsize=2.5pt](12,3.25)
\psline[linearc=0.25,linecolor=black,linewidth=0.2mm](12,4)(12.75,4)
\psdots*[dotsize=2.5pt](12,4)
\psdots*[dotsize=2.5pt](11.25,3.25)
\rput(12.23,3){$\tiny{a_2}$}
\rput(4,4){$A$}
\rput(10,4){$A$}
\rput(12.3,4.2){$\lambda$}
}
\end{pspicture}
\caption{A ball of radius $\lambda$ rolls freely in $A$. For $a_1 \in \partial A$ exists a unique $x\in A$ such that $a_1\in B_{\lambda}[x]\subset A$. For $a_2\in \partial A,\mbox{ }a_2\in B_{\lambda}[x]$ for a non finite number of points $x\in A$.}
\label{Figuraaa}
\end{figure}

The vector $\eta(a)$ established in Lemma \ref{a} is not unique necessarily, see Figure \ref{Figuraaa}. Lemma \ref{A01} relates the uniqueness of this unit vector and the existence of some $x\in A$ such that $a$ coincides metric projection of $x$ onto $A$.

\begin{lemma}\label{A01}Let $A\subset \mathbb{R}^{d}$ be a nonempty and closed set and $a \in \partial A$. Let us assume that there exists $x\notin A$ such that $\rho=\|x-a\|=d(x,A)$, that is, $a$ is a metric projection of $x$ onto $A$. If exists $\lambda>0$ and a unit vector $\eta(a)$ such that
$B_{\lambda}[a-\lambda\eta(a)]\subset A$, then $x=a+\rho\eta(a)$.
\end{lemma}

\begin{proof}To see this suppose the contrary, that is, let us suppose that exists $x$ verifying the required conditions with $x\neq a+\rho\eta(a)$.
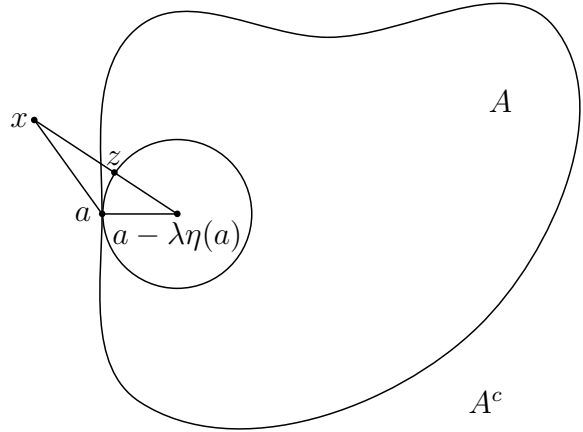
\begin{figure}[h!]\centering
\begin{pspicture}(-2,-.6)(10,5.5)
\psccurve[showpoints=false,fillstyle=solid,fillcolor=white,linecolor=black,linewidth=0.2mm,linearc=3](1.5,0)(1,2.5)(1.5,5)(4,4.85)(7,5)(6,1)
\psline[linearc=0.25,linecolor=black,linewidth=0.2mm](.1,3.75)(1,2.5)
\psdots*[dotsize=2.5pt](.1,3.75)
\rput(-.1,3.75){$\tiny{x}$}
\pscircle[fillstyle=solid,fillcolor=white,linecolor=black,linewidth=0.2mm,linearc=3](2,2.5){1}
\psdots*[dotsize=2.5pt](1,2.5)
\rput(.75,2.5){$\tiny{a}$}
\psdots*[dotsize=2.5pt](2,2.5)
\rput(2,2.2){$\tiny{a-\lambda\eta(a)}$}
\psline[linearc=0.25,linecolor=black,linewidth=0.2mm](1,2.5)(2,2.5)
\psline[linearc=0.25,linecolor=black,linewidth=0.2mm](.1,3.75)(2,2.5)
\psdots*[dotsize=2.5pt](1.165,3.051)
\rput(1.155,3.25){$\tiny{z}$}
\rput(6.3,4){$A$}
\rput(6.1,0){$A^c$}
\end{pspicture}\caption{Elements of Lemma \ref{A01}.}\label{Figuralema}
\end{figure}
$ $\\Then, $x$, $a$ and $a-\lambda\eta(a)$ can not
lie on the same line and hence,
\begin{equation}\label{ecu1}
\|a-\lambda\eta(a)-x\|<\|a-\lambda\eta(a)-a\|+\|a-x\|=\lambda+\rho.\vspace{1.4mm}
\end{equation}Let $z\in \partial B_\lambda[a-\lambda\eta(a)]\cap [x,a-\lambda\eta(a)]$, where $[x,a-\lambda\eta(a)]$ denotes the line segment
with endpoints $x$ and $a-\lambda\eta(a)$ (see Figure \ref{Figuralema}). Then,
$$\|a-\lambda\eta(a)-x\|\leq\|a-\lambda\eta(a)-z\|+\|z-x\|=\lambda+\|z-x\|. $$According to (\ref{ecu1}), $\|z-x\|=\|a-\lambda\eta(a)-x\|-\lambda<\lambda-\rho-\lambda=\rho$, which is a contradiction since $z\in A$ and $\rho=d(x,A)$.\end{proof}
Let us prove that $S=C_r(S)$. Since $S\subset C_r(S)$ for any $r>0$, it is enough to check if $ C_r(S)\subset S$. Equivalently, it will be checked that for all $  x\in S^c$ there exists an open ball of radius $r$ containing $x$. This ball will not intersect $S$. Let us fix $x \notin S$. If $d(x,S)\geq r$ then
$$x\in B_r(x) \mbox{ and } B_r(x)\cap S=\emptyset. $$
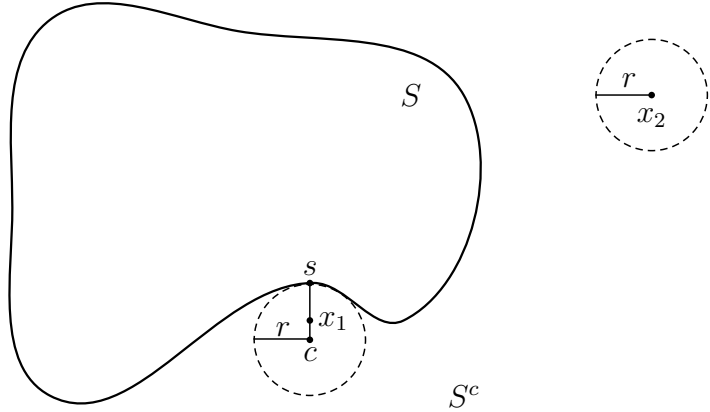
\begin{figure}[h!]
\vspace{-2cm}\centering
\begin{pspicture}(1,-0.5)(10,6.5)
\psccurve[showpoints=false,fillstyle=solid,fillcolor=white,linecolor=black,linewidth=0.3mm,linearc=3](1.5,0)(1,2.5)(1.5,5)(4,4.85)(7,4)(6.2,1)(5,1.5)
 \pscircle[linearc=0.25,linecolor=black,linewidth=0.2mm,linestyle=dashed,dash=3pt 2pt](4.955,.743){0.75}
\psdots*[dotsize=2.5pt](4.955,1.5)
\psdots*[dotsize=2.5pt](4.955, .743)
\psdots*[dotsize=2.5pt](4.955, 1)
\psline[linearc=0.25,linecolor=black,linewidth=0.2mm](4.955,.743)(4.955,1.493)
\psline[linearc=0.25,linestyle=solid,linecolor=black,linewidth=0.2mm](4.955,.7543)(4.205,.7543)
\rput(4.6,.89){$\tiny{r}$}
\rput(5.265,.99){$\tiny{x_1}$}
\rput(4.96,.55){$\tiny{c}$}
\rput(4.96,1.7){$\tiny{s}$}
\rput(6.3,4){$S$}
\rput(7,0){$S^c$}
\pscircle[linearc=0.25,linecolor=black,linewidth=0.2mm,linestyle=dashed,dash=3pt 2pt](9.5,4){0.75}
\psdots*[dotsize=2.5pt](9.5,4)
\rput(9.5,3.7){$\tiny{x_2}$}
\psline[linearc=0.25,linecolor=black,linewidth=0.2mm](9.5,4)(8.75,4)
\rput(9.2,4.2){$\tiny{r}$}
\end{pspicture}\caption{Elements of Proposition \ref{rconvexo} with $d(x_1,S)<r$ and $d(x_2,S)>r$.}\label{ooo}
\end{figure}
$ $\\Otherwise, if $d(x,S)< r$, let $s$ be a projection of $x$ on $S$ and let us define $\rho=d(x,S)=\|x-s\|$. According to Lemmas \ref{a} and \ref{A01},
$$B_{\lambda}[s-\lambda\eta(s)]\subset S, $$where $\eta(s)=(s-x)/\|s-x\|$ and $x=s+\rho\eta(s)$.
In addition, $s\in \partial S $ and, according to the imposed conditions, $S$ fulfills the $r-$rolling property. So,
$$\exists c\in \mathbb{R}^{d} \mbox{ such that }s\in B_r[c] \mbox{ and }B_r(c)\cap S=\emptyset. $$According to Lemma \ref{A01},$$c=s+r\eta(s). $$since $s$ is projection of $c$ on $S$. We are supposing that $\rho<r$. So,
$$\|x-c\|=\|(\rho-r)\eta(s)\|=r-\rho<r .$$Then, $x \notin C_r(S)$ since that $x\in B_r(c) $ and $B_r(c)\cap S=\emptyset$.\\Figure \ref{ooo} shows the elements used in the proof of Proposition \ref{rconvexo}.
 \hfill $\Box$\vspace{.15cm}\\
\emph{Proof of Proposition \ref{maximo}} \vspace{.15cm}\\
It will be proved that $r_0 \in \{\gamma>0:C_\gamma(S)=S\}$. According to the properties of the supreme, $$r_0\in \overline{\{\gamma>0:C_\gamma(S)=S\}} $$and, so,
$$\exists \{r_n\}\subset \{\gamma>0:C_\gamma(S)=S\} \mbox{ such that }\lim_{n\rightarrow \infty} r_n=r_0. $$Without loss of generality, is possible to assume that $\{r_n\}$ is increasing sequence. Then,
$$C_{r_n}(S)=S, \mbox{ } \forall n\in \mathbb{N}. $$Cuevas et al. (2012) proved that $S$ fulfills the $r_n-$rolling property for all $n$. Then, $S$ fulfills the $r_0-$rolling property, see Proposition \ref{nonece}. Taking into account the imposed restrictions, it is verified that $S^c$ satisfies the $\lambda-$rolling condition. So, it is possible to guarantee that $S$ is under $(R_\lambda^{r_0})$. According to Proposition \ref{rconvexo}, $S$ is $r_0-$convex set. Using Proposition 2 in Cuevas et al. (2012), it is possible to guarantee that $S$ fulfills the $r_0-$rolling property.\hfill $\Box$\vspace{.15cm}\\
\emph{Proof of Theorem \ref{consistencia}}\vspace{.15cm}\\
Some auxiliary results are necessary. First we will prove that, with probability increasing to one, $\hat{r}_0$ is at least as big as $r_0$.
\begin{proposition}\label{alpha}
Let $S\subset \mathbb{R}^{d}$ be a compact, nonconvex and nonempty set
verifying ($R_{\lambda}^r$) and $\mathcal{X}_n$ a uniform and i.i.d sample on $S$. Let $r_0$ be the parameter defined in (\ref{maximo2}) and $\{\alpha_n\}\subset (0,1)$ a sequence converging to zero. Then,
$$\lim_{n\to\infty}\mathbb{P}(\hat{r}_0\geq r_0)=1.$$\end{proposition}

\begin{proof}
From the definition of $\hat{r}_0$, see (\ref{r0estimador}), it is clear that
$$\mathbb{P}(\hat{r}_0\geq r_0)\geq \mathbb{P}(\hat{V}_{n,r_0}\leq\hat{c}_{n,\alpha_n,r_0}),$$
where, remember, $\hat{V}_{n,r_0}$ denotes the volume of the maximal spacing in $C_{r_0}(\mathcal{X}_n)$,
$\hat{c}_{n,\alpha_n,r_0}= \mu(C_{r_0}(\mathcal{X}_n))( u_{\alpha_n}+\log{n}+(d-1)\log\log{n}+\log{\beta})\cdot n^{-1}$ and
$u_{\alpha_n}$ satisfies $\mathbb{P}(U\leq u_{\alpha_n})=1-\alpha_n$ and
$U$ is the random variable defined in (\ref{U}).
Since, with probability one, $C_{r_0}(\mathcal{X}_n)\subset S$, we have $\hat{V}_{n,r_0}\leq V_{n}(S)$ where $V_n(S)$ denotes de volume of a ball with radius the maximal spacing
of $S$. Hence,
$$\mathbb{P}(\hat{r}_0\geq r_0) \geq \mathbb{P}(V_{n}(S)\leq \hat{c}_{n,\alpha_n,r_0})=\mathbb{P}\left(\frac{u_{\alpha_n} }{A_n}U_n\leq u_{\alpha_n}\right),$$
where
$$U_n=\frac{nV_{n}(S)}{\mu(S)}-\log{n}-(d-1)\log\log{n}-\log{\beta}$$
and
$$A_n=\frac{n\hat{c}_{n,\alpha_n,r_0}}{\mu(S)}-\log{n}-(d-1)\log\log{n}-\log{\beta}.$$
According to the Janson (1987)'s Theorem, $U_n \stackrel{d}{\rightarrow} U $. In addition, it can be easily proved easily
that $u_{\alpha_n}/A_n\stackrel{P}{\rightarrow} 1$. This can be done by taking into account that
$$\mu(C_{r_0}(\mathcal{X}_n))/\mu(S)=1+O_{P}((\log (n)/n)^{2/(d+1)}),$$
see Theorem 3 in
Rodr\'iguez-Casal (2007). Now, according
to the Slutsky's Lemma, $(u_{\alpha_n}/ A_n)U_n \stackrel{d}{\rightarrow}U $. Notice that $U$ has a continuous distribution, so convergence in distribution implies that
$$
\sup_{u}|\mathbb{P}\left(\left(u_{\alpha_n}/ A_n)U_n\leq u\right)-\mathbb{P}(U\leq u)\right|\to 0.
$$
Since $\mathbb{P}(U\leq u_{\alpha_n})=1-\alpha_n$ and $\alpha_n\to 0$, this ensures
that $$\mathbb{P}\left((u_{\alpha_n} /A_n)U_n\leq u_{\alpha_n}\right)\to 1.$$Therefore, $\mathbb{P}(\hat{r}_0\geq r_0)\to 1. $\end{proof}

It remains to prove that $\hat{r}_0$ cannot be arbitrarily larger that $r_0$. The following lemma ensures that, for a given $\gamma>r_0$,
there exists an open ball contained in $C_{\gamma}(S)$ which does not meet $S$.


\begin{lemma}\label{puntointerior}Let $S\subset \mathbb{R}^{d}$ be a compact, nonconvex and nonempty set verifying ($R_{\lambda}^r$) and let be $\gamma>0$ such
that $S\varsubsetneq C_\gamma(S)$. Then, there exists $\epsilon>0$ and $x\in C_\gamma(S)$ such that $B_\epsilon(x)\subset C_\gamma(S)$ and $B_\epsilon(x)\cap S=\emptyset$.

\end{lemma}

\begin{proof}
Let us assume, for a moment, that we can find $s\in \partial S$ such that $s\in \interior(C_\gamma(S))$. In this case, there exists $\rho>0$ satisfying that $B_{\rho}(s)\subset C_{\gamma}(S)$. On the other
hand, by assumption, $S$ is $r_0-$convex which implies, by Proposition 2 in Cuevas et al. (2012), that $S$ fulfills the $r_0-$rolling condition. This ensures that
there exists a ball $B_{r_0}(y)$ such that $s\in B_{r_0}[y]$ and $B_{r_0}(y)\cap S=\emptyset$. It is clear that we can find an open ball $B_{\epsilon}(x)$ such
that $B_{\epsilon}(x)\subset B_{r_0}(y)\cap B_{\rho}(s)$.
By construction $B_{\epsilon}(x)\subset B_{r_0}(y)$ and, hence, $B_{\epsilon}(x)\cap S=\emptyset$. Finally, $B_{\epsilon}(x)\subset B_{\rho}(s)$ and, therefore, $B_{\epsilon}(x)\subset C_{\gamma}(S)$. This
would finished the proof in this case.

\begin{figure*}
\hspace{.93cm}\begin{pspicture}(-1.3,-0.5)(10,5.9)
\psccurve[showpoints=false,fillstyle=solid,fillcolor=white,linecolor=gray,linewidth=.75mm,linearc=3](1.5,0)(1,2.5)(1.5,5)(4,4.85)(7,4)(6.2,-0.05)(5,1.5)
\psccurve[showpoints=false,fillstyle=solid,fillcolor=white,linecolor=white,linewidth=.85mm,linearc=3](1.5,0)(6.2,-0.05)(5,1.5)
\psarc[showpoints=false,fillstyle=solid,fillcolor=white,linecolor=gray,linewidth=0.45mm,linearc=3](6.108,.7351) {.8}{264}{280}
\psarc[showpoints=false,fillstyle=solid,fillcolor=red,linecolor=gray,linewidth=0.45mm,linearc=3](1.87,0.805) {.9}{240}{281}
\psarc[showpoints=false,fillstyle=solid,fillcolor=white,linecolor=gray,linewidth=0.45mm,linearc=3](4.1,-9.867) {10}{78.9}{102}
\psccurve[showpoints=false,fillstyle=solid,fillcolor=white,linecolor=black,linewidth=0.3mm,linearc=3](1.5,0)(1,2.5)(1.5,5)(4,4.85)(7,4)(6.2,-0.05)(5,1.5)
\psdots*[dotsize=1.75pt](4.8,1.53)
\pscircle[linearc=0.25,linecolor=black,linewidth=0.2mm,linestyle=dashed,dash=3pt 3pt](4.8,1.53){1.}
 \rput(4.8,1.685){\small{$s$}}
 \rput(1.7,4){$S$}
\rput(6.5,5){$C_\gamma(S)$}
\rput(4.8,2.83){\small{$B_{\rho}(s)$}}
\rput(3.72,0.49){\small{$B_{r_0}[y]$}}
 \pscircle[linearc=0.25,linecolor=black,linewidth=0.2mm,linestyle=solid,dash=3pt 2pt](4.79,0.935){.6}
 \pscircle[linestyle=dashed,dash=3pt 3pt,linecolor=gray,linewidth=0.2mm,linearc=0.25,fillstyle=crosshatch*,fillcolor=gray,hatchcolor=white,hatchwidth=1.2pt,hatchsep=.5pt,hatchangle=0](4.795,1.05){.4}
\psdots*[dotsize=1.75pt](4.795,1.05)
\rput(4.795,.85){\small{$x$}}
\end{pspicture}\caption{Elements of proof in Lemma \ref{puntointerior}. $\partial S$ in black, $\partial C_\gamma(S)$ in gray, $B_{\rho}(s)$, $B_{r_0}[y]$ and $B_\epsilon(x)$ in gray.}\label{oossso}
\end{figure*}
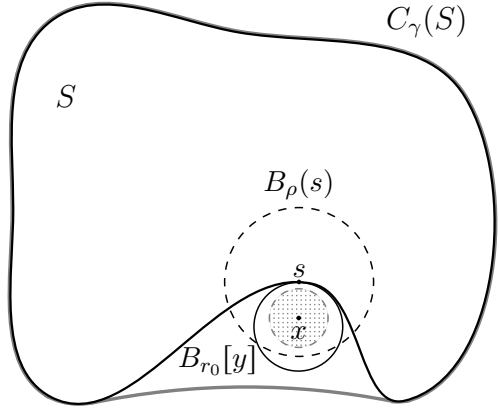

It only remains to show that $\partial S\subset \partial C_{\gamma}(S)$ leads to a contradiction. First, the hypothesis
$\partial S\subset \partial C_{\gamma}(S)$ imply that $S$ satisfy the $\gamma-$rolling condition. This is a straightforward consequence of Proposition 2 in Cuevas et al. (2012)
since $C_{\gamma}(S)$ is $\gamma$-convex. But the $\gamma-$rolling condition imply, under the $(R_{\lambda}^{r})$ shape
restriction, $\gamma$-convexity, see Proposition \ref{rconvexo}.  This is a contradiction since we are assuming that $S\varsubsetneq C_\gamma(S). $  \end{proof}

%
%

\normalsize
\begin{lemma}\label{lem_spacing_minimo}Let $S\subset \mathbb{R}^{d}$ be a compact, nonconvex and nonempty set
verifying ($R_{\lambda}^r$) and $\mathcal{X}_n$ a uniform and i.i.d sample on $S$. Let $r_0$ be the parameter defined in (\ref{maximo2}). Then, for all $r>r_0$, there exists an open ball  $B_\rho(x)$
such that
$B_{\rho}(x)\cap S=\emptyset$ and
 $$
\mathbb{P}\left(B_{\rho}(x)\subset C_{r}(\mathcal{X}_n),\mbox{ eventually}\right)=1
$$

\end{lemma}

\begin{proof}\normalsize
Let be $r^{*}$ such that $r>r^*>r_0$. Since $C_{r_0}(S)=S \subsetneq C_{r^*}(S)$, according to Lemma \ref{puntointerior},
$$\exists B_\epsilon(x) \mbox{ such that }B_\epsilon(x)\subset C_{r^*}(S)\mbox{ and }B_\epsilon(x)\cap S=\emptyset. $$
It can be assumed, without loss
of generality, that $r\leq\frac{\epsilon}{2}+r^*$. If this is not the case then it would be possible to replace $r^*$ by $r^{**}>r^*$ satisfying  $r^{**}<r\leq\frac{\epsilon}{2}+r^{**}$. For this $r^{**}$,
$$B_\epsilon(x)\subset C_{r^*}(S)\subset  C_{r^{**}}(S) \mbox{ and }B_\epsilon(x)\cap S=\emptyset. $$
Now, we can apply Lemma 3 in Walther (1997) in order to ensure that
$$
\mathbb{P}\left(S\oplus r^* B\subset \mathcal{X}_n\oplus rB,\mbox{ eventually}\right)=1.
$$
If $S\oplus r^* B\subset \mathcal{X}_n\oplus rB$ then $(S\oplus r^* B)\ominus r^*B\subset (\mathcal{X}_n\oplus rB)\ominus r^*B$, that is, $C_{r^*}(S)\subset  (\mathcal{X}_n\oplus rB)\ominus r^*B$.
This imply that
$$C_{r^*}(S)\ominus (r-r^*)B\subset ((\mathcal{X}_n\oplus rB)\ominus r^*B)\ominus (r-r^*)B.$$In addition,
$$ ((\mathcal{X}_n\oplus rB)\ominus r^*B)\ominus (r-r^*)B=(\mathcal{X}_n\oplus rB)\ominus rB=C_r(\mathcal{X}_n),$$
where we have used that, for sets $A,C$ and $D$, $(A\ominus C)\ominus D=A\ominus (C\oplus D)$. Finally, since $B_{\epsilon}(x)\subset C_{r^*}(S)$ and
$\epsilon/2\geq (r-r^*)$, we have
$B_{\epsilon/2}(x)\subset C_{r^*}(S)\ominus (\epsilon/2 B)\subset C_{r^*}(S)\ominus (r-r^*)B\subset C_{r}(\mathcal{X}_n)$. This concludes the proof of the lemma by taking $\rho=\epsilon/2$.  \end{proof}

\normalsize
\begin{proposition}\label{prop_nosobreestimamos}
Let $S\subset \mathbb{R}^{d}$ be a compact, nonconvex and nonempty set
verifying ($R_{\lambda}^r$) and $\mathcal{X}_n$ a uniform and i.i.d sample on $S$. Let $r_0$ be the parameter defined in (\ref{maximo2}) and $\{\alpha_n\}\subset (0,1)$ a sequence converging to zero
such that $\log(\alpha_n)/n\to 0$. Then, for any $\epsilon>0$,
$$
\mathbb{P}\left(\hat{r}_0\leq r_0+\epsilon,\mbox{ eventually}\right)=1
$$
\end{proposition}

\begin{proof}\normalsize
Given $\epsilon>0$ let be $r=r_0+\epsilon$. According to Lemma \ref{lem_spacing_minimo}, there exists $x\in\mathbb{R}^d$ and $\rho>0$ such that $B_{\rho}(x)\cap S=\emptyset$
and
$$
\mathbb{P}\left(B_{\rho}(x)\subset C_{r}(\mathcal{X}_n),\mbox{ eventually}\right)=1.
$$
Since, with probability one, $\mathcal{X}_n\subset S$ we have $B_{\rho}(x)\cap \mathcal{X}_n=\emptyset$. Hence, if $B_{\rho}(x)\subset C_{r}(\mathcal{X}_n)$, we
have $\hat{V}_{n,r}\geq \mu(B_{\rho}(x)) =c_{\rho}>0$. Similarly, $\hat{V}_{n,r^\prime}\geq \hat{V}_{n,r}\geq c_{\rho}$ for all $r^{\prime}\geq r$. On the other hand, since
$-u_{\alpha_n}/\log(\alpha_n)=\log(-\log(1-\alpha_n))/\log(\alpha_n)\to 1$, we have, with probability one,
\begin{eqnarray*}
  \sup_{r^{\prime}}\hat{c}_{n,\alpha_n,r^{\prime}} &\leq& \mu(H(S)) (u_{\alpha_n}+\log{n}+(d-1)\log\log{n}+ \\
   &+&\log{\beta})\cdot n^{-1}
\end{eqnarray*}\normalsize{and$$
\mu(H(S))( u_{\alpha_n}+\log{n}+(d-1)\log\log{n}+\log{\beta})\cdot n^{-1}\to 0
$$where $H(S)$ denotes the convex hull of $S$. This means that, with probability one, there is $n_0$ such that if $n\geq n_0$ we have $\sup_{r^{\prime}}\hat{c}_{n,\alpha_n,r^{\prime}}<c_{\rho}$.
Therefore, if $B_{\rho}(x)\subset C_{r}(\mathcal{X}_n)$, we get $\hat{r}_0\leq r$. This last statement follows from
$\hat{V}_{n,r^{\prime}}>\hat{c}_{n,\alpha_n,r^{\prime}}$ for all $r^{\prime}\geq r$ and the definition of $\hat{r}_0$, see (\ref{r0estimador}).
}\end{proof}

\normalsize
Theorem \ref{consistencia} is, then, a straightforward consequence of Propositions \ref{alpha} and \ref{prop_nosobreestimamos}.\hfill $\Box$\vspace{.15cm}\\
\normalsize
\emph{Proof of Theorem \ref{consistencia2}}\vspace{.15 cm}\\For the uniform distribution on $S$, Theorem 3 of Rodr\'{\i}guez-Casal (2007) ensures that, under ($R_{\widetilde{r}}^{\widetilde{r}}$), then $\mathbb{P}(\mathcal{E}_n)\to 1$, where
$$
\mathcal{E}_n=\left\{d_H(S,C_{\widetilde{r}}(\mathcal{X}_n))\leq A\left(\frac{\log n}{n}\right)^{2/(d+1)}\right\},
 $$
and $A$ is some constant. Under the hypothesis of Theorem \ref{consistencia2} this holds for any $\widetilde{r}\leq \min\{r,\lambda\}$. Fix one $\widetilde{r}\leq \min\{r,\lambda\}$ such
that $\widetilde{r}<\nu r_0$ and define
$\mathcal{R}_n=\{\widetilde{r}\leq r_n\leq r_0\}$. Since, by
Theorem \ref{consistencia}, $r_n=\nu \hat{r}_0$ converges in probability to $\nu r_0$ and $\widetilde{r}<\nu r_0<r_0$, we have that $\mathbb{P}(\mathcal{R}_n)\to 1$.  If the events
$\mathcal{E}_n$ and $\mathcal{R}_n$  hold (notice that $\mathbb{P}(\mathcal{E}_n\cap \mathcal{R}_n)\to 1$) we have $C_{\widetilde{r}}(\mathcal{X}_n)\subset C_{r_n}(\mathcal{X}_n)\subset S$ and, therefore,
$$
d_H(S,C_{r_n}(\mathcal{X}_n))\leq d_H(S,C_{\widetilde{r}}(\mathcal{X}_n))\leq A\left(\frac{\log n}{n}\right)^{2/(d+1)}.
$$
This completes the proof of the first statement of Theorem  \ref{consistencia2}. Similarly, it is possible to prove the result for the other error criteria considered in
Theorem \ref{consistencia2}. \hfill $\Box$

\normalsize

$ $\\
\emph{Acknowledgements.} This work has been supported by Project MTM2008-03010 of the Spanish Ministry of Science and Innovation and the IAP network StUDyS (Developing crucial Statistical methods for Understanding major complex Dynamic Systems in natural, biomedical and social sciences) of
Belgian Science Policy.

\end{document}